\documentclass[reqno]{amsart}

\usepackage{external/takodachi}

\usepackage{xcolor}

\newcommand{\sfA}{\mathsf{A}}

\title
{
	Local well-posedness for dispersive equations with bounded data
} 
\author{Jason Zhao}
\address{Department of Mathematics, University of California, Berkeley, 94720}
\email{zhao.j@berkeley.edu}

\date{\today}

\begin{document}

\begin{abstract}
	Given sufficiently regular data \textit{without} decay assumptions at infinity, we prove local well-posedness for non-linear dispersive equations of the form 
		\[
			\partial_t u + \mathsf A(\nabla) u + \mathcal Q(|u|^2) \nabla u= \mathcal N (u, \overline u),
		\]
	where $\mathsf A(\nabla)$ is a Fourier multiplier with purely imaginary symbol of order $\sigma + 1$ for $\sigma > 0$, and polynomial-type non-linearities $\mathcal Q(|u|^2)$ and $\mathcal N(u, \overline u)$. Our approach revisits the classical energy method by applying it within a class of local Sobolev-type spaces $\ell^\infty_{\mathsf A(\xi)} H^s (\mathbb R^d)$ which are adapted to the dispersion relation in the sense that functions $u$ localised to dyadic frequency $|\xi| \approx N$ have size 
		\[
			||u||_{\ell^\infty_{\mathsf A(\xi)} H^s} \approx N^s \sup_{{\operatorname{diam}(Q) = N^\sigma}} ||u||_{L^2_x (Q)}.
		\]	
	In analogy with the classical $H^s$-theory, we prove $\ell^\infty_{\mathsf A(\xi)} H^s$-local well-posedness for $s > \tfrac{d}2 + 1$ for the derivative non-linear equation, and $s > \tfrac{d}2$ without the derivative non-linearity. As an application, we show that if in addition the initial data is spatially almost periodic, then the solution is also spatially almost periodic. 
\end{abstract}

\maketitle

\tableofcontents

\section{Introduction}

In this article we consider local well-posedness for non-linear dispersive equations in a class of bounded continuous functions \textit{without} decay assumptions at infinity. To illustrate the robustness of the argument while keeping the technical details at a minimum, we consider equations of the form 
	\begin{equation}\tag{NL}\label{eq:dispersive}
		\begin{split}
			\partial_t u + \sfA(\nabla) u + \cQ (|u|^2) \nabla u
				&= \cN (u, \overline u),\\
			u_{|t = 0}
				&= u_0,	
		\end{split}
	\end{equation}
where $u: [0, T] \times \R^d \to \C$ is a complex scalar field, $\cQ(|u|^2)$ and $\cN(u, \overline u)$ are polynomials which are either zero or of degree at least one, and $\sfA(\nabla)$ is a Fourier multiplier with purely imaginary symbol $\sfA(\xi) \in C^\infty (\R^d\setminus 0)$ of order $\sigma + 1$ for $\sigma > 0$, i.e. satisfying the derivative bounds 
	\begin{equation} \tag{V}\label{eq:groupvelocity}
		|\nabla_\xi^k \sfA(\xi)| 
                \lesssim_k \langle \xi \rangle^{\sigma + 1 - k},
				\qquad \text{for all integers $k \geq 0$}.
	\end{equation}			
For example, the non-linear 
Schr{\"o}dinger (NLS) and derivative non-linear Schr{\"o}dinger (dNLS) equations are within the scope of \eqref{dispersive} and assumption \eqref{groupvelocity}. After suitable modifications, our arguments also apply to real scalar field solutions $u: [0, T] \times \R^d \to \R$ to equations which preserve reality, in which case we require the symbol of the dispersion relation to satisfy $\overline{\sfA(\xi)} = \sfA(-\xi)$ and the non-linearities are instead polynomials $\cQ(u)$ and $\cN(u)$ in the variable $u$. The Korteweg-de Vries (KdV), Benjamin-Ono (BO) and intermediate long waves (ILW) equations fall within this category. 

Initial data of particular interest, and the equations they have been posed on, include 
\begin{itemize}
	\item almost periodic functions on the line $\R$, e.g.
		\[
		u_0 (x) 
			= \cos(x) + \cos(\sqrt 2 x),
		\]
		for NLS \cite{Monvel,Oh2014,Oh2015a,Schippa2024} and KdV \cite{Egorova1994, DamanikGoldstein2015,BinderEtAl2018,Tsugawa2012,Schippa2024}, and general dispersive equations \cite{Papenburg2024},

	\item functions on the plane $\R^2$ with non-zero degree at infinity, e.g.
		\[
			\lim_{r \to \infty} u_0 (r e^{i \theta}) 
				=
				e^{i n \theta}, \qquad n \in \Z, 
		\]
		for the Gross-Pitaevskii equation \cite{BethuelSmets2007}, 

	\item perturbations of a fixed background with exotic spatial asymptotics, e.g. 
		\[
			u_0 (x)
				= \operatorname{tanh}(x) + q_0 (x), \qquad \lim_{|x| \to \infty} q_0(x) = 0,
		\]
		for the KdV equation \cite{IorioEtAl1998, Gallo2005, Laurens2023}, gKdV \cite{Palacios2022}, and BO \cite{IorioEtAl1998, Gallo2005}. 
		
	\item functions with continuous bounded derivatives up to order $k$, for a regularised NLS \cite{DodsonEtAl2020}.
\end{itemize}

The main novelty of this article is in studying the initial data problem \eqref{dispersive} within a uniform local Sobolev space adapted to the dispersion relation $\ell^\infty_{\sfA(\xi)}  H^s (\R^d)$. At first approximation, one can regard our function space as embedded in the local Sobolev space $H^s_\loc (\R^d)$ and such that the projection $P_N u$ to a dyadic frequency $N \in 2^\N$ has size 
	\[
		||P_N u||_{\ell^\infty_{\sfA(\xi)}  H^s }\approx N^s \sup_{\substack{Q \subseteq \R^d \text{ cube} \\  \operatorname{diam}(Q) = N^\sigma}} ||P_N u||_{L^2_x (Q)}.
	\]
The norm above is adapted to the dispersion relation in the sense that wave packets localised to frequency $|\xi| \approx N$ travel with speed $|\nabla_\xi \sfA(\xi)| \lesssim N^\sigma$ under the linear equation, 
	\begin{equation}\label{eq:lindispersive}\tag{L}
		\begin{split}
		\partial_t u + \sfA(\nabla) u 
			&= 0,
		\end{split}
	\end{equation}
so on unit time scale $t \approx 1$ we expect each wave packet to have traveled within a cube of length $L \approx N^\sigma$. This heuristic suggests then that the linear flow approximately conserves the $\ell^\infty_{\sfA(\xi)}  H^s$-norm on unit time scales, i.e. we expect the energy estimate
	\begin{equation}\label{eq:prop1}\tag{$\star$}
		||e^{-t \sfA(\nabla)} u_0 ||_{\ell^\infty_{\sfA(\xi)}  H^s} \lesssim ||u_0||_{\ell^\infty_{\sfA(\xi)}  H^s}, \qquad t \in [0, 1].
	\end{equation}
Furthermore, the space $\ell^\infty_{\sfA(\xi)} H^s (\R^d)$ forms a Banach algebra when $s > \tfrac{d}{2}$,
	\begin{equation}\label{eq:prop2}\tag{$\star\star$}
		||uv||_{\ell^\infty_{\sfA(\xi)} H^s} \lesssim ||u||_{\ell^\infty_{\sfA(\xi)} H^s}||v||_{\ell^\infty_{\sfA(\xi)} H^s}.
	\end{equation}

Using the analogues of \eqref{prop1}-\eqref{prop2} for the usual Sobolev spaces $H^s$, it is well-known in the folklore that one can apply the abstract semigroup method of Kato \cite{Kato1993} to conclude $H^s$-local well-posedness for equations of the form \eqref{dispersive} at regularity $s > \tfrac{d}{2}$ when there is no derivative in the non-linearity, i.e. $\cQ(|u|^2) \equiv 0$, and $s > \tfrac{d}{2} + 1$ with derivative non-linearity, i.e. $\cQ(|u|^2) \not\equiv 0$. The latter case relies additionally on the energy method to overcome what one naively expects to be a derivative loss in the energy estimates, see \cite{BonaScott1976, TsutsumiFukuda1980,AbdelouhabEtAl1989} for examples on the line. Our first main result follows by adapting these arguments to the $\ell^\infty_{\sfA(\xi)} H^s$-setting,

\begin{theorem}[$\ell^\infty_{\sfA(\xi)}  H^s$-local well-posedness]\label{thm:lwp}
	The non-linear dispersive equation \eqref{dispersive} with dispersion relation obeying \eqref{groupvelocity} 
	is $\ell^\infty_{\sfA(\xi)} H^s$-locally well-posed for $s > \tfrac{d}2 + 1$ in the following sense: 
	\begin{enumerate}
		\item \label{eq:exist} existence; for initial data $u_0 \in \ell^\infty_{\sfA(\xi)} H^s (\R^d)$, there exists a solution $u \in C^0_t (\ell^\infty_{\sfA(\xi)} H^s)_x ([0, T] \times \R^d)$ to the equation \eqref{dispersive} up to a time $T := T(||u_0||_{\ell^\infty_{\sfA(\xi)} H^s})$ depending on the size of the data,
		
		\item \label{eq:unique} weak Lipschitz dependence; for solutions $u, v \in C^0_t (\ell^\infty_{\sfA(\xi)} H^s)_x([0, T] \times \R^d)$ to the equation \eqref{dispersive}, we have
			\[
				||u - v||_{C^0_t (\ell^\infty_{\sfA(\xi)} H^0)_x} \lesssim_{||u_0||_{\ell^\infty_{\sfA(\xi)}  H^s}, ||v_0||_{\ell^\infty_{\sfA(\xi)}  H^s}} ||u_0 - v_0||_{\ell^\infty_{\sfA(\xi)} H^0},
			\]

		\item \label{eq:ctsdependence}continuous dependence; given a sequence of solutions $\{u_j\}_j \subseteq  C^0_t (\ell^\infty_{\sfA(\xi)} H^s)_x ([0, T] \times \R^d)$ and a solution $u \in C^0_t (\ell^\infty_{\sfA(\xi)} H^s)_x ([0 ,T] \times \R^d)$ to \eqref{dispersive} arising respectively from initial data $\{u_{0j}\}_j \subseteq \ell^\infty_{\sfA(\xi)} H^s (\R^d)$ and $u_0 \in \ell^\infty_{\sfA(\xi)} H^s (\R^d)$, we have
		\[
			\lim_{j \to \infty}||u - u_j||_{C^0_t (\ell^\infty_{\sfA(\xi)} H^s)_x} = 0 \qquad \text{whenever } \lim_{j \to \infty} ||u_0 - u_{0j}||_{\ell^\infty_{\sfA(\xi)} H^s} = 0.
		\]
	\end{enumerate}	
	When the derivative in the non-linearity is absent, $\cQ(|u|^2) \equiv 0$, then \eqref{dispersive} is locally well-posed for $s > \tfrac{d}{2}$, and furthermore the data-to-solution map satisfies
	\begin{enumerate}\setcounter{enumi}{3}
		\item \label{eq:stronglip}strong Lipschitz dependence; for solutions $u, v \in C^0_t (\ell^\infty_{\sfA(\xi)} H^s)_x ([0, T] \times \R^d)$
			\[
				||u - v||_{C^0_t (\ell^\infty_{\sfA(\xi)} H^s)_x} \lesssim_{||u_0||_{\ell^\infty_{\sfA(\xi)} H^s}, ||v_0||_{\ell^\infty_{\sfA(\xi)} H^s}} ||u_0 - v_0||_{\ell^\infty_{\sfA(\xi)} H^s}.
			\]
	\end{enumerate}
\end{theorem}

\begin{remark}
	The analogous theorem also holds in the class of real-valued $\ell^\infty_{\sfA(\xi)} H^s$-functions for equations which preserve reality, e.g. the KdV, BO, and ILW equations. 
\end{remark}

\begin{remark}
	The uniqueness statement \eqref{unique} of Theorem \ref{thm:lwp} is interesting when compared with other local well-posedness results. For initial data with prescribed spatial asymptotics, e.g. vanishing at infinity or quasi-periodicity, solutions are typically constructed within a class of functions with the same spatial asymptotics. For example, an $H^s$-solution is only shown to be unique when compared against other competitors in $H^s$, where $H^s := H^s (\R^d)$ or $H^s := (\TT^d)$. In view of the embeddings 
		\begin{align*}
			H^s (\R^d) 
				&\hookrightarrow \ell^\infty_{\sfA(\xi)}  H^s(\R^d),\\
			H^{s + \frac{d}{2}\sigma} (\TT^d)
				&\hookrightarrow \ell^\infty_{\sfA(\xi)} H^s (\R^d),
		\end{align*}
	our result shows that these $H^s$-solutions are unique even when compared to competitors with \textit{different} spatial asymptotics. In this regard, Theorem \ref{thm:lwp} \eqref{unique} has a similar spirit to the recent unconditional uniqueness result of Chapouto-Killip-Vișan for bounded solutions of the KdV equation \cite{ChapoutoEtAl2024}. 
\end{remark}

\begin{remark}
	The Sobolev embeddings \eqref{sobolev0}-\eqref{sobolev} state that for $l > s + \tfrac{d}{2}\sigma$ and $s > k + \tfrac{d}{2}$ we have
		\[
			C^l (\R^d) \hookrightarrow\ell^\infty_{\sfA(\xi)}  H^s(\R^d) \hookrightarrow C^k (\R^d).
		\]
	Thus the existence statement \eqref{exist} of Theorem \ref{thm:lwp} shows that initial data $u_0 \in C^l (\R^d)$ admits a Cauchy development under the equation \eqref{dispersive}. To our knowledge, this is the first local existence result on the scale of $C^k$-spaces for equations of the form \eqref{dispersive}. 
	
	The closest result we know of is the existence result for the following regularised NLS on $\R$,
		\[
			\partial_t u + i \partial_x^2 u = \pm i P_{\leq M} (|P_{\leq M} u|^2 P_{\leq M} u).
		\]
	with $C^2$-data by Dodson-Soffer-Spencer \cite{DodsonEtAl2020}. They prove linear $C^k$-estimates, with derivative loss, using stationary phase, while the truncation of the non-linearity in frequency space allows one to appeal to finite speed of propagation arguments. 
	While the equation above is not of the form \eqref{dispersive}, our arguments can be easily modified to recover their result, since $C^2 (\R) \hookrightarrow \ell^\infty_{\sfA(\xi) = |\xi|^2} H^{1/2+}(\R)$. 
\end{remark}

Our second main result concerns the propagation of almost periodicity under the equation \eqref{dispersive}, which was our original impetus for pursuing Theorem \ref{thm:lwp}. We restrict ourselves to one space dimension $d = 1$, as it is the setting which has garnered the most attention in the context of dispersive equations, see e.g. \cite{Tsugawa2012, Oh2014,Oh2015a, DamanikGoldstein2015,BinderEtAl2018, Papenburg2024, Schippa2024} and the references therein, though one can easily generalise the notion of almost periodicity to higher dimensions. 

To motivate the problem, let us speak loosely; we say $u: \R \to \C$ is almost periodic if it takes the form 
	\[
	u(x) = \sum_{\lambda \in \Lambda} \widehat u (\lambda) e^{i \lambda x},
	\]
for a countable set of frequencies $\sigma(u)\subseteq \R$ and Fourier coefficients $\widehat u(\lambda) \in \C$. Then, continuing to ignore issues of convergence, observe that the class of almost periodic functions is preserved by the linear flow \eqref{lindispersive} and the non-linearity of \eqref{dispersive}. Thus one might expect that almost periodic initial data leads to almost periodic solutions to \eqref{dispersive}. We provide a rigorous proof of this statement for a class of almost periodic data within the scope of Theorem \ref{thm:lwp}. More precisely, define the space of (Bohr) almost periodic functions $\mathtt{AP} (\R)$ as the uniform closure of finite trigonometric polynomials, 
	\[
		\mathtt{AP} (\R) := \overline{\Big\{ \sum_{\lambda \in \Lambda} a_\lambda e^{i \lambda x} : \text{finite $\Lambda \subseteq \R$ and $a_\lambda \in \C$}  \Big\}}^{L^\infty_x}.
	\]
An almost periodic function $u \in \mathtt{AP} (\R)$ admits a countable set of frequencies $\sigma(u) \subseteq \R$, known as the spectrum, given by  
	\[
		\sigma (u) 
			:= \Big\{ \xi \in \R : \lim_{L \to \infty} \frac{1}{2L} \int_{-L}^L u(x) e^{i x \xi} \, dx\neq 0 \Big\}.
	\] 
Then, as a corollary of the iteration argument used to prove Theorem \ref{thm:lwp} and Sobolev embedding \eqref{sobolev}, 

\begin{theorem}[Propagation of almost periodicity]\label{thm:AP}
	Let $u_0 \in (\mathtt{AP} \cap \ell^\infty_{\sfA(\xi)} H^s) (\R)$ be almost periodic for $s > \tfrac32$, then the solution to the initial data problem \eqref{dispersive} is also almost periodic $u \in C^0_t (\mathtt{AP} \cap \ell^\infty_{\sfA(\xi)} H^s)_x ([0, T] \times\R)$. Furthermore, the spectrum remains in the integer span of the initial spectrum,
		\[
			\sigma(u(t)) \subseteq \Big\{ \sum_{j} n_j \lambda_j : \text{finitely-many non-zero $n_j \in \Z$ and $\lambda_j \in \sigma(u_0)$} \Big\}.
		\]
	When there is no derivative in the non-linearity, $\cQ(|u|^2) \equiv 0$, the result holds for $s > \tfrac12$. 
\end{theorem}

\begin{remark}
	Our definition of almost periodicity differs from Bohr's original formulation \cite{Bohr1925} in terms of almost periods. Nonetheless, the two are equivalent, see e.g. \cite[Chapter 3.2]{Corduneanu2009}. 
\end{remark}

\begin{remark}
	The well-posedness results for almost periodic data in \cite{Tsugawa2012, DamanikGoldstein2015,Oh2015a,BinderEtAl2018,Papenburg2024,Schippa2024} all assume decay or summability assumptions on the Fourier coefficients. Aside from \cite{Oh2015a}, these results also assume \textit{quasi-periodicity}, i.e. the spectrum is finitely generated over $\Z$. In contrast, 
	Theorem \ref{thm:AP} proves local well-posedness for almost periodic functions with \textit{arbitrary} spectrum and \textit{no decay or summability assumptions} on the Fourier coefficients. 

	The results in \cite{Egorova1994,Monvel, Oh2014} work with norms which are of a similar flavour to ours, namely 
		\[
			||u||_{S^{s, p}} 
				:= \sup_{y \in \R} ||u||_{H^s ([y, y + 1])}. 
		\]
	\cite{Egorova1994} for KdV and \cite{Monvel, Oh2014} for NLS prove global existence with data which can be approximated, with exponential convergence in the $S^{s, 2}$-norm for $s$ sufficiently large, by periodic functions with growing periods $\alpha_j \to \infty$. 
\end{remark}

\begin{remark}
	While the study of global existence for \eqref{dispersive} goes beyond the scope of this article, let us motivate the problem for the interested reader. Given plane wave initial data, the corresponding evolution under the linear equation \eqref{lindispersive} is also a plane wave oscillating in both space and time, 
	\[
		e^{-t \sfA(\partial_x)} e^{i \xi_0 x} = e^{-i t \sfA(\xi_0)} \, e^{i \xi_0 x}.
	\]
	Following our earlier heuristic discussion, one might ask whether a global solution $u: \R\times \R \to \C$ to \eqref{dispersive} with almost periodic initial data $u_0 : \R \to \C$ is almost periodic in space-time, taking the form 
		\[
			u(t, x) = \sum_{\substack{\lambda \in \Lambda \\ \mu \in \Pi}} \widehat u(\mu,\lambda) \, e^{i \mu t} \, e^{i \lambda x},
		\]
	for some coefficients $\widehat u(\mu, \lambda) \in \C$ and countable frequency sets $\Lambda, \Pi \subseteq \R$. 
	
	In the setting of spatially periodic solutions to KdV, McKean-Trubowitz \cite{McKeanTrubowitz1976} answered this question in the affirmative, showing that $C^\infty (\TT)$ solutions evolve almost periodically in time using complete integrability, with subsequent extensions by Bourgain \cite{Bourgain1993} to $L^2 (\TT)$ and Kappeler-Topalov \cite{KappelerTopalov2006} to $H^{-1} (\TT)$. Inspired by these results, Deift \cite{Deift2008,Deift2017} conjectured that one could relax the assumptions on the initial data to almost periodicity. 
	Damanik-Goldstein in \cite{DamanikGoldstein2015}, joined later by Binder and Lukic in \cite{BinderEtAl2018}, confirmed the conjecture for small amplitude quasi-periodic analytic initial data with Diophantine frequencies. 
	\end{remark}

	\begin{remark}
		For KdV, it is interesting to compare Theorem \ref{thm:AP} to \cite[Theorem 1.3]{ChapoutoEtAl2024} which exhibits a bounded solution with almost periodic data $u_0 \in \mathtt{AP} (\R)$ which loses continuity at a later time $t_* > 0$, so in particular $u(t_*) \not\in \mathtt{AP} (\R)$ and the Deift conjecture fails for KdV in the class $\mathtt{AP} (\R)$. 

	\end{remark}

\subsection*{Outline of the paper}

After recalling some conventional notation in Section \ref{sec:prelim}, the main substance of the paper begins in Section \ref{sec:function}, where we define the $\ell^\infty H^s$-spaces and establish the analogues of the standard Sobolev space estimates. Section \ref{sec:linear} concerns the linear energy estimate \eqref{prop1}. In Section \ref{sec:nonlinear}, we carry out the energy method to prove \textit{a priori} estimates for \eqref{dispersive}, establishing an energy estimate and weak Lipschitz dependence (Theorem \ref{thm:lwp} \eqref{unique}). We start Section \ref{sec:existence} by introducing the semigroup method, proving local well-posedness without derivative non-linearity. The remainder of the section finishes off the existence theory (Theorem \ref{thm:lwp} \eqref{exist}) by constructing a solution to the derivative non-linear equation via approximate solutions. We conclude continuous dependence (Theorem \ref{thm:lwp} \eqref{ctsdependence}) in Section \ref{sec:cts} by the method of frequency envelopes. In Section \ref{sec:periodic}, we show propagation of almost periodicity (Theorem \ref{thm:AP}).

\begin{acknowledgements}
	The author would like to thank David Ambrose for discussions which inspired us to initiate this work, Sung-Jin Oh for suggesting the scale of spaces developed in \cite{MarzuolaEtAl2012}, Sultan Aitzhan for discussions on the context of Theorem \ref{thm:AP} relative to the existing literature, and Robert Schippa for pointing out that the energy method fails when the derivative non-linearity is not gauge covariant, and in fact the equation can be ill-posed. During the writing of this paper, the author was partially supported by the National Science Foundation CAREER Grant under NSF-DMS-1945615.

\end{acknowledgements}

\section{Preliminaries}\label{sec:prelim}

\subsubsection*{Asymptotic notation}

The notations $X \lesssim Y$ and $X = O(Y)$ denote the inequality $X \leq CY$ for an implicit constant $C > 0$. Moreover $X \sim Y$ is an abbreviation for $X \lesssim Y \lesssim X$. We indicate the dependence of implicit constants on parameters by subscripts, e.g. $X \lesssim_s Y$ denotes $X \leq C(s) Y$. 

\subsubsection*{Dyadic integers}

We take the convention that the natural numbers contain zero, that is, $\N := \{ 0, 1, 2, \dots \}$ and denote by $2^\N = \{ 1, 2, 4, \dots\}$ for the dyadic integers. Often times we will write upper-case letters, such as $N, M, A, B$, to represent a dyadic integer. When working with frequency envelopes (see Section \ref{sec:cts}), it will be convenient to write lower-case letters for the dyadic exponent, e.g. $N = 2^n$ and $M = 2^m$. 

\subsubsection*{An $\epsilon$ of room} Given $s \geq 0$, we write $s+ := s + \epsilon$ and $s- := s - \epsilon$ for any $\epsilon > 0$. We will suppress the dependence of implicit constants on $\epsilon$, though one should keep in mind they blow-up when taking $\epsilon \to 0$. 

\subsubsection*{Fourier transform and multipliers}

For a tempered distribution $f \in \cS(\R^d)^*$, we denote its Fourier transform by $\widehat u$ and its inverse Fourier transform by $\widecheck f$. Given a tempered distribution $\mathsf m : \R^d \to \C$, we define the Fourier multiplier $\mathsf m(\nabla) : \cS (\R^d) \to \cS(\R^d)^*$ by the formula 
	\[
		\widehat{\mathsf{m} (\nabla) f}(\xi) = \mathsf{m}(\xi) \widehat f(\xi).
	\]
We refer to $\mathsf{m}(\xi)$ as the symbol of the multiplier $\mathsf{m} (\nabla)$. 	

\subsubsection*{Littlewood-Paley projections}

Fix $\phi \in C^\infty_c (\R^d)$ that is non-negative, radial, and such that $\phi(\xi) \equiv 1$ if $|\xi| \leq 1$ and $\phi(\xi) \equiv 0$ if $|\xi| \geq \tfrac{11}{10}$. For each dyadic integer $N \in 2^\N$, set 
	\begin{align*}
		\psi_N (\xi)
			&= 
			\begin{cases} 
				\phi\big( \tfrac\xi N \big) - \phi \big( \tfrac{2\xi}{N} \big),
					&\text{if $N \geq 2$},\\
				\phi(\xi),
					&\text{if $N = 1$},	
			\end{cases}	 \\
		\phi_N (\xi)	
			&= \phi\big(\tfrac{\xi}{N}\big),
	\end{align*}
By construction, $\{\psi_N\}_N$ form a dyadic partition of unity for frequency space. Define the Littlewood-Paley projections to frequencies $|\xi| \sim N$ and $|\xi| \lesssim N$ respectively as the Fourier multipliers
	\begin{align*}
		\widehat{P_N f} (\xi) 
			&:= \psi_N (\xi)\widehat f (\xi),\\
		\widehat{P_{\leq N} f} (\xi)
			&:= \phi \big(\tfrac{\xi}{N}\big) \widehat f (\xi),
	\end{align*}
We will frequently abbreviate the projections by $f_N := P_N f$ and $f_{\leq N} := P_{\leq N} f$. 

The name "projection" is a bit of a misnomer; the multipliers $P_N$ fail to be true projections in the sense that $P_N P_N \neq P_N$. Nevertheless, a slightly modified statement holds; define the fattened projections by
	\[
		\widetilde{P_N} := P_{\frac{N}{2}} + P_N + P_{2N},
	\]
then $\widetilde{P_N} P_N = P_N$.

The projections are bounded on $L^p_x(\R^d)$ (and in fact on any translation-invariant Banach space by Minkowski's inequality; see Proposition \ref{prop:bounded}) and obey, for $1 \leq p \leq q \leq \infty$, Bernstein's inequalities, 
\begin{align*}
	||f_N||_{L^q}
		&\lesssim_{p, q} N^{\frac{d}p - \frac{d}q} ||f||_{L^p},\\
	||f_{\leq N}||_{L^q}
		&\lesssim_{p, q} N^{\frac{d}p - \frac{d}q} ||f_{\leq N}||_{L^p},	
\end{align*}
and the Sobolev-Bernstein inequalities, 
		\begin{alignat*}{2}
			|| \nabla^k f_{\leq N} ||_{L^p}
				&\lesssim N^k ||f_{\leq N}||_{L^p},
				&& \\
			|| \nabla^k f_N||_{L^p}	
				&\sim N^k ||f_N ||_{L^p}, && \qquad \text{for $N \geq 2$}.
		\end{alignat*}

\section{Uniform local Sobolev space}\label{sec:function}

Assuming the dispersion relation is a symbol of order $\sigma + 1$, i.e. satisfying \eqref{groupvelocity}, wave packets localised at frequency $|\xi| \approx N$ travel with group velocity $|\nabla_\xi \sfA(N)| \lesssim N^\sigma$, and thereby within spatial cubes of length $L \approx N^\sigma$ on unit time scales. Thus, we expect a translation-invariant norm capturing this phenomenon to be approximately conserved by the linear flow for short times. 

To build such a norm, fix a cut-off function $\chi \in C^\infty_c (\R^d)$ that is non-negative, radially decreasing, and such that $\chi \equiv 1$ on the cube $[-\tfrac14, \tfrac14]^d$ and vanishing outside the cube $[-\tfrac34, \tfrac34]^d$. Recentering on the lattice $j \in \Z^d$ and rescaling by a dyadic integer $N \in 2^{\N}$, denote
	\[
		\chi_{j, N} (x) := \chi\Big( \tfrac{x - j}{N^\sigma} \Big).
	\]
We will often suppress the subscripts for brevity in cases where the center $j$ is not important in view of translation-invariance or when the scale $N$ is clear from the context. After possibly rescaling the amplitude on the overlapping regions, we can choose our cut-off such that for each $N$ the family of translates $\{ \chi_{j, N}\}_j$ forms a partition of unity subordinate to a cover by cubes of length scale $L \approx N^\sigma$,  
	\[
		1 \equiv \sum_{j \in \Z^d} \chi_{j, N}. 
	\]
We measure waves localised to frequency $|\xi| \approx N$ by the norm 
	\[
		||f||_{{(\ell^\infty_{\sfA(\xi)} L^2)_N}}
			:= \sup_{j \in \Z^d}|| \chi_{j, N} \,  f||_{L^2_x}.
	\]
Decomposing an arbitrary tempered distribution $f \in \cS(\R^d)^*$ into dyadic frequencies and square summing these norms adapted to each Littlewood-Paley piece, we obtain the \textit{uniform local (inhomogeneous) Sobolev norm} adapted to the dispersion relation $\sfA(\xi)$, 
	\[
		||f||_{\ell^\infty_{\sfA(\xi)} H^s}
			:= \Big|\Big| N^{s} || P_N f||_{(\ell^\infty_{\sfA(\xi)} L^2)_N} \Big|\Big|_{\ell^2_N}.
	\]
We denote the space of tempered distributions with finite norm by $\ell^\infty_{\sfA(\xi)} H^s (\R^d)$. As we will later see, one can alternatively think of the space as the completion of the space of smooth functions with bounded derivatives with respect to the $\ell^\infty_{\sfA(\xi)} H^s$-norm. For the remainder of this article, we will fix a dispersion relation and, for brevity, suppress the dependence on $\sfA(\xi)$, writing $\ell^\infty H^s \equiv \ell^\infty_{\sfA(\xi)} H^s$.

\begin{remark}
	The uniform local Sobolev space can be viewed as the $p = \infty$ endpoint in the family of Sobolev-type spaces $\ell^p H^s (\R^d)$. Indeed, our construction is inspired by previous work on the $p = 1$ endpoint in the context of Schr{\"o}dinger equations, $\sfA(\xi) = i |\xi|^2$, by Marzuola-Metcalfe-Tataru \cite{MarzuolaEtAl2012} and KdV-type equations, $\sfA(\xi) = - i \xi^3$, by Harrop-Griffiths 	\cite{Harrop-Griffiths2015,Harrop-Griffiths2015a}. In these works the motivation was to develop a function space of \textit{spatially-localised} initial data with a \textit{translation-invariant} norm in which one can leverage \textit{local smoothing} to prove local well-posedness for highly non-linear equations. 
\end{remark}

\subsection{Basic properties}\label{subsec:properties}

For completeness, we record some properties of the $(\ell^\infty L^2)_N$- and $\ell^\infty H^s$-norms which might as well be regarded as folklore results. 

\begin{proposition}[Boundedness of $P_N$]\label{prop:bounded}
	For each dyadic integer $N, M \in 2^\N$, the projections $P_N$ and $P_{\leq N}$ are bounded with respect to the $(\ell^\infty L^2)_M$-norm,  
		\begin{equation}\label{eq:bounded}
			\begin{split}
				||P_N f||_{(\ell^\infty L^2)_M} 
					&\lesssim || f ||_{(\ell^2_j L^2_x)_M},\\
				||P_{\leq N} f||_{(\ell^\infty L^2)_M} 
					&\lesssim || f ||_{(\ell^2 L^2)_M}.  
			\end{split} 
		\end{equation}
	As a corollary, the projections are also bounded with respect to the $\ell^\infty H^s$-norms.
\end{proposition}

\begin{proof}
	We prove boundedness of $P_N$ with respect to $(\ell^\infty L^2)_M$; the proof for $P_{\leq N}$ is similar, and boundedness on $\ell^\infty H^s$ is immediate from the definition. A change of variables shows that the kernel of $P_N$ satisfies $||\widecheck{\psi_N}||_{L^1_x} \sim ||\widecheck \psi||_{L^1_x} \sim 1$. We conclude from Minkowski's inequality and translation invariance \eqref{equiv1} that
		\begin{align*} 
			|| \chi P_N f||_{L^2_x}
				&\leq \int_{\R^d} \widecheck{\psi_N} (y) || \chi(x) f (x - y)||_{L^2_x} \, dy \lesssim \sup_{j \in \Z^d} ||\chi_j f||_{L^2_x}.
		\end{align*}	
	This completes the proof. 	
\end{proof}

In the proof above, we used the fact that the $(\ell^\infty L^2)_M$ norms are equivalent to translation-invariant norms. Indeed, we will often find it convenient to modify our definitions of the $(\ell^\infty L^2)_N$- and $\ell^\infty H^s$-norms to more suitable, yet nonetheless equivalent, norms.

\begin{proposition}[Equivalent norms]\label{prop:equiv}
	The $(\ell^\infty L^2)_N$- and $\ell^\infty H^s$-norms can be equivalently defined up to the following modifications,  
	\begin{enumerate}
		\item replacing supremum over $j \in \Z^d$ with supremum over $y \in \R^d$, furnishing a translation-invariant norm,
			\begin{equation}
				||f||_{(\ell^\infty L^2)_N} \sim \sup_{y \in \R^d}||\chi_{y, N} \,  f ||_{L^2_x},\label{eq:equiv1}
			\end{equation}	

		\item replacing sharp cut-offs $\chi$ with frequency-localised cut-offs $P_{\leq 1} \chi$, 
		\begin{equation}\label{eq:equiv2}
				||f||_{(\ell^\infty L^2)_N} \sim \sup_{j \in \Z^d} || P_{\leq 1}\chi_{j, N} \, f ||_{L^2_x},
		\end{equation}

		\item replacing projections $P_N$ with fattened projections $\widetilde{P_N}$,
			\begin{equation}
				||f||_{\ell^\infty H^s}
					\sim \Big|\Big| N^s ||\widetilde{P_N} f ||_{(\ell^\infty L^2)_N}  \Big|\Big|_{\ell^2_N}\label{eq:equiv3}.
			\end{equation}
	\end{enumerate}	
	Furthermore, the $(\ell^\infty L^2)_N$-norms are equivalent, satisfying 
	\begin{enumerate}
		\setcounter{enumi}{3}
		\item monotonicity in $N$, 
			\begin{equation}\label{eq:monotone}
				||f||_{(\ell^\infty L^2)_N}
					\lesssim 
					\begin{cases}
						||f||_{(\ell^\infty L^2)_M}, 
						&\text{if $M \geq N$,}\\
						\left( \frac{N}{M}\right)^{d \sigma} ||f||_{(\ell^\infty L^2)_M},
						&\text{if $M < N$}.
					\end{cases}
			\end{equation}
	\end{enumerate}	
\end{proposition}

\begin{proof}
\leavevmode
\begin{enumerate}
	\item The translation invariant norm trivially bounds the non-invariant norm from above. For the reverse inequality, observe that an arbitrary cube of side length $2N^\sigma$ can be covered by at least $O(d)$ neighboring cubes of side length $N^\sigma$ centered on the lattice $N^\sigma \Z^d$. It follows that 
		\begin{align*}
			\sup_{y \in \R^d} ||| \chi_{y, N} \, f||_{L^2_x} 
				&\lesssim_d \sup_{j \in \Z^d} ||\chi_{j, N} \, f ||_{L^2_x},
		\end{align*}
	as desired. 
	
	\item Observe that a symmetric proof to that of boundedness of the projections \eqref{bounded} implies the left-hand side of \eqref{equiv2} controls the right-hand side. For the reverse inequality, observe that the kernel of $P_{\leq 1}$ is non-negative, so we can estimate pointwise $\chi \lesssim P_{\leq 1} \chi$. 
		
	\item Since $P_N$ are bounded with respect to these norms \eqref{bounded} and $P_N \widetilde{P_N} = P_N$, the right-hand side of \eqref{equiv3} controls the left-hand side. For the reverse inequality, we decompose $\widetilde{P_N} = P_{N/2} + P_N + P_{2N}$ and apply the triangle inequality.  
	
	\item For $M \geq N$, this follows immediately from $\chi_N \leq \chi_M$ and translation invariance \eqref{equiv1}. The case $M < N$ follows from the fact that each cube of side length $L \approx N^\sigma$ can be partitioned into $O((N/M)^{d \sigma})$-many 
	cubes of side length $L \approx M^\sigma$. 
\end{enumerate}
\end{proof}

\begin{remark}
	One can also use sharp cut-offs $\chi = \mathbb 1_{[-1, 1]^d}$ as indicated in the introduction, though it will be convenient for our commutator estimates to use smooth cut-offs. 
\end{remark}

\begin{proposition}[Basic embeddings]\label{prop:embed}
	For $s \geq 0$, the following embeddings hold:
	\begin{enumerate}	
		\item \label{rellich} Monotonicity of $\ell^\infty H^s$ in $s$, i.e. for $s < s'$, we have $\ell^\infty H^{s'} (\R^d) \hookrightarrow \ell^\infty H^s (\R^d)$ and 
			\begin{equation}
				||f||_{\ell^\infty H^s}
			\lesssim ||f||_{\ell^\infty H^{s'}}.
			\end{equation}
			
		\item Embedding into the space of tempered distributions, $	\ell^\infty H^s (\R^d) \hookrightarrow \cS(\R^d)^*$. 		
	\end{enumerate}
\end{proposition}

\begin{proof}
\leavevmode
	\begin{enumerate}
		\item This follows immediately from $N^s \leq N^{s'}$ for all dyadic integers $N \in 2^\N$. 
		
		\item Consider the action of $f \in \ell^\infty H^s (\R^d)$ as a distribution on Schwartz functions $\phi \in \cS (\R^d)$. Decomposing in frequency space $f = \sum_N f_N$ and then in physical space $f_N = \sum_j \chi_{j, N} P_N f$, we write
			\begin{align*}
				\langle f, \phi \rangle
					= \sum_{j \in \Z^d} \sum_{N \in 2^\N} \langle \chi_{j, N}  f_N , \, \widetilde{\chi_{j, N}} \phi \rangle,
			\end{align*}
		where $\widetilde{\chi_{j, N}} \in C^\infty_c (\R^d)$ are fattened smooth cut-offs such that $\widetilde{\chi_{j, N}} \equiv 1$ on the support of $\chi_{j, N}$ and vanishing outside a fattened cube of side length $L \approx N^\sigma$. It follows that 
			\begin{align*}
				|\langle f, \phi \rangle|
					&\lesssim 
						\sum_{j \in \Z^d} \sum_{N \in 2^\N} || \chi_{j, N}  f_N ||_{L^2_x} ||\widetilde{\chi_{j, N}} \phi ||_{L^2_x}
					\\
					&\lesssim 
						\sum_{j \in \Z^d} \sum_{N \in 2^\N} \left(N^{-s} ||f||_{\ell^\infty H^s} \right) \left( \langle jN^2 \rangle^{-100d} || \langle x \rangle^{100d}\phi||_{L^\infty} \right) \lesssim || \langle x \rangle^{100d}\phi||_{L^\infty}  ||f||_{\ell^\infty H^s}.
			\end{align*}
	\end{enumerate}
\end{proof}

\begin{remark}
	The analogue of the Rellich-Kondrachov theorem would say that the embedding in (\ref{rellich}) is compact. Unfortunately, this is false, precisely because the $\ell^\infty H^s$-spaces are non-separable. As an explicit example, consider the sequence $\{ \exp(i \lambda_n \cdot x) \}_n \subseteq C^\infty (\R^d)$ for frequencies satisfying $\tfrac14 \leq |\lambda_n| \leq \tfrac12$. Then 
		\[
			||\exp(i \lambda_n \cdot x) ||_{\ell^\infty H^s} \sim 1,
		\]
	uniformly in $s$ and $n$, and 
		\[
			||\exp(i \lambda_n \cdot x) - \exp(i \lambda_m \cdot x) ||_{\ell^\infty H^s} \sim ||1 - \exp(i (\lambda_n - \lambda_m) \cdot x)||_{\ell^\infty H^s} \sim 1
		\]
	whenever $\lambda_n \neq \lambda_m$. Thus there is no way to extract a subsequence converging in any $\ell^\infty H^s$-norm. 	
\end{remark}

Finally, we will need a bound on the commutator between a physical multiplier $f(x)$ and frequency multiplier $\mathsf m(\nabla)$. The principal symbol of the commutator $[\mathsf m(\nabla), f(x)]$ is given precisely by the Poisson bracket $\{\mathsf m(\xi), f(x)\}$, so to leading order 
	\[
		[\mathsf m(\nabla), f(x)] \approx (\nabla_x f)(x) \cdot (\nabla_\xi \mathsf m)(\nabla). 
	\]
When the frequency multiplier is taken to be a Littlewood-Paley projection $P_N$, we obtain the Littlewood-Paley product rule 
$[P_N, f] g = O(\nabla f \, N^{-1}g)$ originally due to Tao \cite[Lemma 2]{Tao2001}. This will be useful for moving a derivative from a function $g$ localised to high frequencies $|\xi| \sim N$ onto a function $f$ localised to low frequencies $|\xi| \lesssim N$.

\begin{proposition}[Commutator bound]
	Let $N, M \in 2^\N$ be dyadic integers, then 
		\begin{equation}\label{eq:commutator0}
			|| [\mathsf m(\nabla), f] g||_{(\ell^\infty L^2)_M} 
				\lesssim ||\widecheck{\nabla_\xi \mathsf m}||_{L^1_x} ||\nabla_x f||_{L^\infty_x} ||g||_{(\ell^\infty L^2)_M}. 
		\end{equation}
	In particular, 
	\begin{align}\label{eq:commutator}
		||[P_N, f] g||_{(\ell^\infty L^2)_M} 
			\lesssim N^{-1} ||\nabla f||_{L^\infty_x} || g||_{(\ell^\infty L^2)_M}.
	\end{align}
\end{proposition}

\begin{proof}
	We can write the commutator as 
		\begin{align*}
			[\mathsf m(\nabla), f] g (x) 
				&= \int_{\R^d} \widecheck{\mathsf m}(y) \big( f(x - y) - f(x) \big) g(x - y) \, dy\\
				&=-\int_{\R^d} \widecheck{\mathsf m}(y) y \cdot \Big(  \int_0^1 \nabla_x f (x -\theta y) \, d\theta \Big) g(x - y)\, dy\\
				&=- \int_{\R^d} \widecheck{\nabla_\xi \mathsf m} (y)\cdot \Big(  \int_0^1 \nabla_x f (x -\theta y) \, d\theta \Big) g(x - y)\, dy,
		\end{align*}
	where the second line follows from the fundamental theorem of calculus, and the third line follows from properties of the Fourier transform. Placing $\nabla f$ in $L^\infty_x$ and $g$ in $(\ell^\infty L^2)_M$ using Minkowski's integral inequality and translation invariance \eqref{equiv1}, we conclude 
		\[
			||\chi [\mathsf m(\nabla), f] g||_{L^2_x} 
				\lesssim ||\widecheck{\nabla_\xi \mathsf m}||_{L^1_x} ||\nabla f||_{L^\infty_x} || g||_{(\ell^\infty L^2)_M}. 
		\]
	Since $\chi :=\chi_{j, M}$ was arbitrary, this completes the proof. 	
\end{proof}

\subsection{Sobolev embeddings}

Denote by $C^k (\R^d)$ the space of functions with continuous bounded derivatives up to order $k$. It forms a Banach space with respect to the norm
	\[
		||f||_{C^k} := \sum_{j = 0}^k ||\nabla^k u ||_{L^\infty_x}.
	\]
We aim to show the following Sobolev-type embeddings, 
	\[
		C^l (\R^d) \hookrightarrow \ell^\infty H^s (\R^d) \hookrightarrow C^k (\R^d)
	\]	
for $l > s + \tfrac{d}{2} \sigma$ and $s > k + \tfrac{d}{2}$. The former is reminiscent of the embedding on the torus $C^l (\TT^d) \hookrightarrow H^l (\TT^d)$, while the latter is the analogue of the classical Sobolev embedding $H^s (\R^d) \hookrightarrow C^k (\R^d)$. To prove the latter, we will need the analogues of the Bernstein and Sobolev-Bernstein inequalities for the $(\ell^\infty L^2)_N$-norms,

\begin{lemma}
	Let $N, M \in 2^\N$ be dyadic integers.
		\begin{enumerate}
			\item (Bernstein's inequality) We have
			\begin{equation}\label{eq:bernstein}
				||f_N||_{L^\infty_x} 
					\lesssim N^{\frac{d}2} || f_N ||_{(\ell^\infty L^2)_M}.
			\end{equation}
			
			\item (Sobolev-Bernstein inequality) For each $k \in \N$, we have 
			\begin{equation}\label{eq:sobolevbernstein}
				||\nabla^k f_N||_{(\ell^\infty L^2)_M} \lesssim N^k || f_N ||_{(\ell^2_j L^2_x)_M}.
			\end{equation}
		In particular, $\nabla^k$ is a bounded map from $\ell^\infty H^{s + k} (\R^d)$ to $\ell^\infty H^s (\R^d)$. 	
		\end{enumerate}
\end{lemma}

\begin{proof}
	To apply the usual Bernstein and Sobolev-Bernstein inequalities, we use frequency-localised cut-offs as in \eqref{equiv2} rather than sharp cut-offs. Abusing notation, we write $\chi := P_{\leq 1} \chi_{j, M}$. 
	\begin{enumerate}
		\item By translation invariance, we can estimate 
			\begin{align*}
				||f_N||_{L^\infty_x} 
					&\lesssim \sup_{j \in \Z^d} || \chi f_N ||_{L^\infty_x}. 
			\end{align*}
		The product $\chi f_N$ has frequency support in $|\xi| \lesssim N$, so by Bernstein's inequality, we can estimate the right-hand side by 
			\[
				|| \chi f_N ||_{L^\infty_x}
					\lesssim N^{\frac{d}2} ||f_N ||_{(\ell^\infty L^2)_M}. 
			\]
		
		\item It suffices by induction to prove the result for $k = 1$. We commute the derivative with the cut-off and apply the triangle inequality to write 
		\begin{align*}
			||\chi \nabla f_N||_{L^2_x} 
				\leq ||\nabla (\chi f_N)||_{L^2_x} + ||[\chi, \nabla] P_N f||_{L^2_x}.
		\end{align*}
	The product $\chi f_N$ has frequency support in $|\xi| \lesssim N$, so by the Sobolev-Bernstein inequality	
		\[
			||\nabla (\chi f_N)||_{L^2_x} \lesssim N ||\chi f_N ||_{L^2_x} \lesssim N ||f_N||_{(\ell^\infty L^2)_M} .
		\]
	The commutator $[\chi, \nabla] = \nabla\chi$	contributes amplitude $|\chi'| \leq M^{-\sigma} \leq 1$ and is localised in space to neighboring cubes of side length $L \approx M^\sigma$, so its contribution is of lower order, 
		\[
			||[\chi, \nabla] f_N ||_{L^2_x} 
				\lesssim \sum_{k = j + O(1)} ||\chi_k f_N ||_{L^2_x} \lesssim N ||f_N ||_{(\ell^\infty L^2)_M} .
		\]
	This completes the proof. 	
	\end{enumerate}
\end{proof}

\begin{proposition}[Sobolev embedding]
	\leavevmode 
	\begin{enumerate}
		\item For $l > s + \tfrac{d}{2} \sigma$, we have $C^l (\R^d) \hookrightarrow \ell^\infty H^s (\R^d)$ and 
		\begin{equation}
			||f||_{\ell^\infty H^s} \lesssim ||f||_{C^l}. \label{eq:sobolev0}
		\end{equation}

		\item For $s > k + \tfrac{d}{2}$, we have $\ell^\infty H^s (\R^d) \hookrightarrow C^k (\R^d)$ and 
		\begin{equation}\label{eq:sobolev}
			||f||_{C^k} 
				\lesssim ||f||_{\ell^\infty H^s}.
		\end{equation}
	\end{enumerate}
\end{proposition}

\begin{proof}
\leavevmode 
	\begin{enumerate}
		\item For low frequencies $N = 1$, we have the trivial bound 
		\[
			||\chi_j \, f_1||_{L^2_x}
				\leq ||\chi||_{L^2_x} ||f||_{L^\infty_x}. 	
		\]
	For high frequencies $N \geq 2$, it follows from scaling and the Sobolev-Bernstein inequality that 
		\begin{align*} 
			||\chi_{j, N} \, f_N||_{L^2_x}
				&\leq || \chi_{N} ||_{L^2_x} ||f_N||_{L^\infty_x} \\
				&\lesssim N^{\frac{d}{2} \sigma} || \chi||_{L^2_x} N^{- k} || \nabla^k f||_{L^\infty}.
		\end{align*}	
	Multiplying both sides by $N^s$, we see that the right-hand side is square-summable in $N \in 2^\N$ provided that $k > s + \frac{d}{2}\sigma$. Collecting these inequalities, we conclude	
		\begin{align*}
			||f||_{\ell^\infty H^s} 
				\lesssim || N^{s -k + \frac{d}{2} \sigma} ||_{\ell^2_N} ||\chi||_{L^2_x} \left( || f||_{L^\infty_x} + ||\nabla^k f||_{L^\infty_x} \right)\lesssim_{k, s} ||f||_{C^k}.
	\end{align*}

	\item It suffices by Sobolev-Bernstein \eqref{sobolevbernstein} to prove the result for $k = 0$. 
	Decomposing in frequency space $f = \sum_N  f_N$, we can estimate 
		\begin{align*}
			||f||_{L^\infty} 
				&\leq\sum_{N \in 2^{\N}} ||f_N||_{L^\infty}\\
				&\lesssim \sum_{N \in 2^\N} N^{\frac{d}2} ||f_N||_{(\ell^\infty L^2)_N} \\
				&\lesssim || N^{\frac{d}2} N^{-s} ||_{\ell^2_N} ||f||_{\ell^\infty H^s} \lesssim_s ||f||_{\ell^\infty H^s},
		\end{align*}
	applying the triangle inequality in the first line, Bernstein's inequality \eqref{bernstein} in the second, and Cauchy-Schwartz in the third.
	\end{enumerate}
\end{proof}

The proof of \eqref{sobolev} actually only furnishes an \textit{a priori} estimate; to show that $f \in \ell^\infty H^s (\R^d)$ is indeed continuously differentiable up to order $k$, we need an approximation by smooth functions. The natural candidates are the localisations to low frequency $\{f_{\leq N}\}_N \subseteq C^\infty (\R^d)$, which are in fact real analytic by the Paley-Wiener theorem. 

\begin{corollary}[Smooth approximation]
	Let $f \in \ell^\infty H^s (\R^d)$, then $f_{\leq N} \in C^\infty (\R^d)$ and 
		\begin{equation}\label{eq:approx}
			\lim_{N \to \infty} ||f_{\leq N} - f||_{\ell^\infty H^s} = 0.
		\end{equation}
	In particular, the space of smooth functions with bounded derivatives $C^\infty (\R^d)$ is dense in $\ell^\infty H^s (\R^d)$.
\end{corollary}

\begin{proof}
	We know that $f_{\leq N} \in \ell^\infty H^s (\R^d)$ for every $s \geq 0$ since it has compact frequency support. It follows from Paley-Wiener and Sobolev embedding \eqref{sobolev} that it is smooth with bounded derivatives $f_{\leq N} \in C^\infty (\R^d)$. On the other hand, $f_{\leq N}- f$ has frequency support in $|\xi| \gtrsim N$, so	
		\[
			||f_{\leq N} - f||_{\ell^\infty H^s} \leq \left|\left| A^s ||u_A||_{(\ell^\infty L^2)_A}  \right|\right|_{\ell^2_A (A \gtrsim N)} \overset{N \to \infty}{\longrightarrow} 0
		\]
	by monotone convergence. 	
\end{proof}

\subsection{Paradifferential calculus}\label{sec:paraproduct}
	
To handle the non-linear terms in the equation \eqref{dispersive}, we will need to record some bilinear and non-linear estimates. The statements and proofs are simply adaptations of the standard ones on $L^p_x$-spaces, as detailed in e.g. \cite[Appendix A]{Tao2006}, to the $\ell^\infty H^s$-spaces. Indeed, we begin with the well-known the Littlewood-Paley trichotomy, which asserts that if $f_A$ and $g_B$ have frequency support on $|\xi| \sim A$ and $|\xi| \sim B$ respectively, then the projection of their product $P_N (f_A g_B)$ to frequency $|\xi| \sim N$ vanishes unless one of the following holds:
	\begin{itemize}
		\item \textit{low-high interactions:} $f_A$ has significantly lower frequency that $g_B$, with $B \sim N$ and $A \ll B$,
		\item \textit{high-low interactions:} $f_A$ has significantly higher frequency that $g_B$, with $A \sim N$ and $B \ll A$,
		\item \textit{high-high interactions:} $f_A$ and $g_B$ have comparable frequencies, with $A \sim B$ and $A, B \gtrsim N$. 
	\end{itemize}
In accord with the trichotomy, we can decompose the projection of a generic product $P_N (f g)$ into a low-high, high-low, and high-high paraproducts, 
	\begin{align*}
		P_N (fg) 
			&= \sum_{A, B \in 2^\N} P_N (P_A f\, P_B g)\\
			&= P_N \Big( P_{\leq \frac{N}{128}} f\, \widetilde{P_N} g\Big) + P_N \Big( \widetilde{P_N} f \, P_{\leq \frac{N}{128}} g\Big) + \sum_{\substack{A \sim B \\ A, B > \frac{N}{128}}} P_N (P_A f \, P_B g) \\
			&=: \big(\mathsf T_f g \big)_N + \big( \mathsf T_g f)_N + \Pi (f, g)_N,
	\end{align*}
with suitable modifications in the low frequency range $N = 1, 2, 4, \dots, 128$ since we use inhomogeneous Littlewood-Paley projections. 

\begin{proposition}[Endpoint Coifman-Meyer estimates]
	The paraproducts satisfy 
		\begin{align}
			||\mathsf T_f g ||_{\ell^\infty H^0}
				&\lesssim ||f||_{L^\infty_x} ||g||_{\ell^\infty H^0},\label{eq:lowhigh1} \\ 
			||\mathsf T_f g ||_{\ell^\infty H^0}
				&\lesssim ||f||_{\ell^\infty H^0} ||g||_{\ell^\infty H^{\frac{d}2}}\label{eq:lowhigh2},\\
			||\Pi(f, g)||_{\ell^\infty H^0} 
				&\lesssim ||f||_{\ell^\infty H^0} ||g||_{\ell^\infty H^{\frac{d}2+}}\label{eq:highhigh}.	
		\end{align}
	The analogous estimates to \eqref{lowhigh1} and \eqref{lowhigh2} hold for the high-low paraproduct $T_g f$. Collecting these inequalities, the following product estimate holds,
		\begin{align}\label{eq:product0}
			||fg||_{\ell^\infty H^0} \lesssim ||f||_{\ell^\infty H^0} ||g||_{\ell^\infty H^{\frac{d}2+}}.
		\end{align}	
\end{proposition}

\begin{proof}[Proof of \eqref{lowhigh1}]
	The estimate is immediate after placing the low frequencies in $L^\infty_x$, 
		\begin{align*}
			\Big|\Big|P_N \Big( P_{\leq \frac{N}{128}} f \widetilde{P_N} g\Big) \Big|\Big|_{(\ell^\infty L^2)_N}
				&\lesssim ||f||_{L^\infty_x} ||\widetilde{P_N} g||_{(\ell^\infty L^2)_N},
		\end{align*}
	and then square summing in $N \in 2^\N$. 
	\end{proof}

	\begin{proof}[Proof of \eqref{lowhigh2}]
	We decompose the low frequency term $f_{\leq N/128} = \sum_M f_M$, placing each piece into $L^\infty_x$ and then passing to $(\ell^\infty L^2)_M$ using Bernstein's inequality \eqref{bernstein}, and finally applying Cauchy-Schwartz in $M \leq \tfrac{N}{128}$, 
		\begin{align*}
			\Big|\Big|P_N \Big( P_{\leq \frac{N}{128}} f \widetilde{P_N} g\Big) \Big|\Big|_{(\ell^\infty L^2)_N}
				&\lesssim ||g_N||_{(\ell^\infty_jL^2_x)_N} \sum_{M \leq \frac{N}{128}} M^{\frac{d}2}  ||f_M ||_{(\ell^\infty L^2)_M} \lesssim  ||g_N||_{(\ell^\infty_jL^2_x)_N} N^{\frac{d}2} ||f||_{\ell^\infty H^0}.
		\end{align*}	
Square summing in $N \in 2^\N$ finishes the proof. 
\end{proof}

\begin{proof}[Proof of \eqref{highhigh}]
We again place one of the high frequencies in $L^\infty_x$ and pass to $(\ell^\infty L^2)_N$ via Bernstein's inequality \eqref{bernstein}, and finally apply Cauchy-Schwartz in $A$, 
		\begin{align*}
			\Big|\Big| \sum_{\substack{A \sim B\\ A, B > \frac{N}{128}}} P_N \big( f_A g_B \big) \Big|\Big|_{(\ell^\infty L^2)_N} 
				&\lesssim \sum_{\substack{A \sim B \\ A, B > \frac{N}{128}}} A^{\frac{d}2} ||f_A||_{(\ell^\infty L^2)_A} ||g_B||_{(\ell^\infty L^2)_B} \lesssim N^{0-} ||f||_{\ell^\infty H^{\frac{d}2+}} ||g||_{\ell^\infty H^0}. 
		\end{align*}
Square summing in $N \in 2^\N$, we conclude the result. 
\end{proof}

When applying a differential operator $D^s$ of positive order $s > 0$ to a paraproduct, heuristically one expects the dominant contribution to come from when all the derivatives fall on the high frequencies,
	\begin{alignat*}{2}
		D^s (f_A g_B) 
			&\approx f_A (D^s g_B)
			&&\qquad \text{when $A \ll B$},\\
		D^s (f_A g_B)							&\lessapprox (D^s f_A) g_B \approx f_A (D^s g_B)
			&&\qquad \text{when $A \sim B$}.
	\end{alignat*}

\begin{proposition}[Product rule]
	Let $s > 0$, then we have the estimate 
		\begin{equation}\label{eq:product}
			||fg||_{\ell^\infty H^s} 
				\lesssim_s ||f||_{\ell^\infty H^s} ||g||_{L^\infty_x} + ||f||_{L^\infty_x} ||g||_{\ell^\infty H^s},
		\end{equation}
	for all $f,g \in \ell^\infty H^s (\R^d) \cap L^\infty_x (\R^d)$. In particular, $\ell^\infty H^s(\R^d)$ is an algebra if $s > \tfrac{d}2$. 
\end{proposition}

\begin{proof}
	Again we decompose $P_N (fg)$ into low-high, high-low, and high-high interactions. Following the proof of the low-high paraproduct estimate \eqref{lowhigh1}, we see that the unbalanced frequency interactions can be estimated as 
		\begin{align*}
			||\mathsf T_f g||_{\ell^\infty H^s} 
				&\lesssim ||f||_{L^\infty_x} ||g||_{\ell^\infty H^s},\\
			||\mathsf T_g f||_{\ell^\infty H^s}	
				&\lesssim ||g||_{L^\infty_x} ||f||_{\ell^\infty H^s}.
		\end{align*}
	In contrast to the $s = 0$ high-high estimate \eqref{highhigh}, we can leverage decay of high frequencies when $s > 0$ to gain a sharper estimate. Indeed, placing $g_B$ in $L^\infty_x$ and making a change of variables $A = CN$, 
	\begin{align*}
		N^s \Big|\Big| \sum_{\substack{A \sim B\\ A, B > \frac{N}{128}}} P_N \big( f_A g_B \big) \Big|\Big|_{(\ell^\infty L^2)_N} 
			&\lesssim ||g||_{L^\infty_x} \sum_{\substack{A > \frac{N}{128}}} N^s ||f_A||_{(\ell^\infty L^2)_A} \\
			&\lesssim ||g||_{L^\infty_x}  \sum_{C > \frac{1}{128}} C^{-s} (CN)^s ||f_{CN}||_{(\ell^\infty_jL^2_x)_{CN}}. 
	\end{align*}
	Applying the triangle inequality after square summing in $N \in 2^\N$, we conclude 
		\[
			||\Pi(f, g)||_{\ell^\infty H^s} 
				\lesssim ||f||_{\ell^\infty H^s} ||g||_{L^\infty_x}. 
		\]
	Adding up the contributions from each paraproduct, we conclude \eqref{product}. 	
\end{proof}

\begin{corollary}[Chain rule for polynomials]
	Let $\cP(z,\overline z)$ be a polynomial of degree $p \in \N$, then 
	\begin{enumerate}
		\item for $s > 0$, the Sobolev bounds, 
		\begin{equation} \label{eq:chain1}
			||\cP(u, \overline u)||_{\ell^\infty H^s} 
				\lesssim ||u||_{L^\infty_x}^{p - 1} ||u||_{\ell^\infty H^s},
		\end{equation}
		and, for $s = 0$, the analogous bounds after replacing the $L^\infty_x$-norm with an $\ell^\infty H^{\frac{d}2+}$-norm, 
		\begin{equation}
			||\cP(u, \overline u)||_{\ell^\infty H^0} 
				\lesssim ||u||_{\ell^\infty H^{\frac{d}{2}+}}^{p - 1} ||u||_{\ell^\infty H^0},
		\end{equation}

		\item for $s > 0$, the difference bounds, 
		\begin{equation} \label{eq:chain2}
			||\cP(u, \overline u) - \cP(v, \overline v) ||_{\ell^\infty H^s} 
				\lesssim \left( ||u||_{L^\infty_x}^{p - 1} + ||v||_{L^\infty_x}^{p - 1} \right) ||u - v||_{\ell^\infty H^s} + \left( ||u||_{\ell^\infty H^s}^{p - 1} + ||v||_{\ell^\infty H^s}^{p - 1} \right) ||u - v||_{L^\infty_x},
		\end{equation}		
		and, for $s = 0$, the Lipschitz bounds,
		\begin{equation}\label{eq:chain3}
			||\cP(u, \overline u) - \cP(v, \overline v) ||_{\ell^\infty H^0}
				\lesssim \Big( ||u||_{\ell^\infty H^{\frac{d}2+}}^{p - 1} + ||v||_{\ell^\infty H^{\frac{d}2+}}^{p - 1} \Big) ||u - v||_{\ell^\infty H^0} .
		\end{equation}
	\end{enumerate}
\end{corollary}

\begin{proof}
	The Sobolev bounds \eqref{chain1} follow by inductively applying either the product rule \eqref{product} in the $s > 0$ case or the Coifman-Meyer estimate \eqref{product0} in the $s = 0$ case. To prove the difference bounds \eqref{chain2}-\eqref{chain3}, we use the fundamental theorem of calculus to write 
		\[
			\cP(u, \overline u) - \cP(v, \overline v) = \int_0^1 (u - v) \cdot \partial_z \cP(w_\theta, \overline{w_\theta}) + (\overline u - \overline v) \cdot \partial_{\overline z} \cP(w_\theta, \overline{w_\theta}) \, d \theta
		\]
	where $w_\theta := \theta u + (1 - \theta) v$. Since $\partial_z \cP(z, \overline z)$ and $\partial_{\overline z} \cP(z, \overline z)$ are polynomials of degree $p - 1$, the difference bounds \eqref{chain2} follow from the product rule \eqref{product} and the Sobolev bounds \eqref{chain1}, while the Lipschitz bounds \eqref{chain3} follow from the product estimate \eqref{product0} and Sobolev bounds \eqref{chain1}. 
\end{proof}

\section{Linear estimates}\label{sec:linear}
Recall from the discussion in Section \ref{sec:function} that the $\ell^\infty H^s$-norm was constructed such that, heuristically, it should be approximately conserved by the linear flow \eqref{lindispersive} on unit time scales. In particular, we expect that the linear propagator $e^{- t \sfA(\nabla)}$ is a bounded operator on $\ell^\infty H^s (\R^d)$ for each $t \in \R$. Indeed, we claim that the linear flow satisfies the following energy estimate,

\begin{theorem}[Linear propagator bound]
	For $s \geq 0$, the solution to the linear equation \eqref{lindispersive} with initial data $u_0 \in \ell^\infty H^s (\R^d)$ satisfies 
		\begin{equation}\label{eq:linearenergy}
			||e^{- t \sfA(\nabla)} u_0 ||_{\ell^\infty H^s}^2 
				\leq || u_0 ||_{\ell^\infty H^s}^2 \exp(Ct)
		\end{equation}
	for some uniform constant $C > 0$. 	
\end{theorem}

    In view of how the $\ell^\infty H^s$-norm is constructed, we first recast the equation \eqref{dispersive} into a paradifferential equation for $\chi u_N$. Applying the projection $P_N$ to the equation \eqref{lindispersive} and multiplying by the cut-off $\chi$, we see that $\chi u_N$ satisfies the equation 
\begin{equation}
	\partial_t (\chi u_N) + \sfA(\nabla) (\chi u_N) + [\chi, \sfA(\nabla)] u_N = 0. 
\end{equation}
We proceed by the energy method, multiplying by $\overline{\chi u_N}$, integrating over $[0, T] \times \R^d$, and taking the real part. As per the usual proof, the first two terms are 
\begin{align*}
	\Re \int_0^T \int_{\R^d} \overline{\big( \chi u_N\big)} \partial_t \big(\chi u_N \big) \, dx dt
		&= \int_0^T\partial_t \Big( \int_{\R^d} \frac12 \big|\chi u_N \big|^2 \, dx\Big)dt\\
		&= \int_{\R^d} \frac12 \big|\chi u_N (T) \big|^2 \, dx - \int_{\R^d} \frac12 \big|\chi u_N (0) \big|^2 \, dx,\\
	\Re\int_0^T \int_{\R^d} \overline{\big( \chi u_N \big)} \sfA(\nabla) \big( \chi u_N \big) \, dx dt 
		&= 0.
\end{align*}
The first identity follows from the fundamental theorem of calculus, the second identity follows from Plancharel's theorem and the symbol of $\sfA(\nabla)$ being purely imaginary. Moving the initial data term and the commutator term to the right-hand side, applying Cauchy-Schwartz and taking supremum over $\chi$ on both sides, we obtain 
	\[
		\frac12 ||u_N(T)||_{(\ell^\infty L^2)_N}^2
			\leq \frac12 ||u_N (0)||_{(\ell^\infty L^2)_N}^2 + \int_0^T ||u_N||_{(\ell^\infty L^2)_N} || [\chi, \mathsf A(\nabla)] u_N ||_{L^2_x} \, dt
	\]
To apply Gronwall's inequality and conclude the linear bound \eqref{linearenergy}, it remains to show the commutator term obeys the estimate 
    \begin{equation}\label{eq:commsymb}
        || [\chi, \mathsf A(\nabla)] u_N ||_{L^2_x} 
            \lesssim ||u_N||_{(\ell^\infty L^2)_N}.
    \end{equation}

To prove \eqref{commsymb}, it is convenient to assume the cut-off is frequency-localised, i.e. $\chi := P_{\leq 1} \chi$, then we can write $[\chi, \mathsf A(\nabla)] u_N = [\chi, (\mathsf A \psi_{\frac{N}{2}})(\nabla)] u_N + [\chi, (\mathsf A \psi_N)(\nabla)] u_N + [\chi, (\mathsf A\psi_{2N})(\nabla)] u_N$, where recall that $\psi_N$ is the symbol of the Littlewood-Paley projection $P_N$. Following the proof of the commutator bound \eqref{commutator0}, 
	\begin{align*}
		|| [\chi_{j, N}, (\mathsf A \psi_N)(\nabla)] u_N ||_{L^2_x} 
				&\leq \int_{\R^d} \big|\widecheck{\nabla_\xi (\mathsf A \psi_N)} (y)\big|\cdot \Big(  \int_0^1 N^{-\sigma} || \nabla_x \chi \big(\tfrac{x -\theta y - j}{N^\sigma}\big) u_N(x - y)||_{L^2_x}\, d\theta \Big) \, dy\\
				&\leq N^{-\sigma} || \widecheck{\nabla_\xi (\mathsf A \psi_N)} ||_{L^1_x}  \sup_{y \in \R^d} \sup_{\theta \in [0, 1]} || \nabla_x \chi \big(\tfrac{x -\theta y - j}{N^\sigma}\big) u_N(x - y)||_{L^2_x}\\
				&\lesssim N^{-\sigma} || \widecheck{\nabla_\xi (\mathsf A \psi_N)} ||_{L^1_x}   || u_N||_{(\ell^\infty L^2)_N}
	\end{align*}
	where in the last line we have used the translation-invariance \eqref{equiv1}  and the fact that $\nabla_x \chi$ is localised to neighboring cubes at scale $N^\sigma$. This reduces the problem down to showing that
        \begin{equation}\label{eq:kernel}
            || \widecheck{\nabla_\xi (\mathsf A \psi_N)} ||_{L^1_x} 
                \lesssim N^\sigma. 
        \end{equation}

    It follows from scaling under the Fourier transform that the kernel can be rewritten as 
    \[
        \widecheck{\nabla_\xi (\sfA \psi_N)} (x) =N^d \widecheck{((\nabla_\xi \sfA)_{\frac1N} \psi)} (Nx)+ N^{d - 1} \widecheck{(\sfA_{\frac1N} \nabla_\xi \psi)} (Nx),
    \]
where we denote rescaling by $f_\lambda(\xi) := f(\xi/\lambda)$. By the triangle inequality and a change of variables, 
    \[
        ||  \widecheck{\nabla_\xi (\sfA \psi_N)}||_{L^1_x}
            \leq ||\widecheck{(\nabla_\xi \sfA)_{\frac1N} \psi}||_{L^1_x} + N^{-1}||\widecheck{\sfA_{\frac1N} \nabla_\xi \psi}||_{L^1_x}. 
    \]
We estimate the first term on the right in $L^1_x$, the second term is similar. Our strategy is to trade the derivative bounds on the symbol \eqref{groupvelocity} for decay. 
Let $m \in \N$ be an integer such that $2m > \tfrac{d}{2}$, then 
    \begin{align*}
        ||\widecheck{(\nabla_\xi \sfA)_{\frac1N} \psi}||_{L^1_x}
            &\leq ||(1 + |x|^{2m})^{-1}||_{L^2_x} || (1 + |x|^{2m}) \widecheck{((\nabla_\xi \sfA)_{\frac1N} \psi)}||_{L^2_x}\\
            &\lesssim ||{(\nabla_\xi \sfA)_{\frac1N}} \psi ||_{L^2_\xi} + \big|\big|\nabla^{2m}_\xi \big( {(\nabla_\xi \sfA)_{\frac1N} \psi}\big) \big|\big|_{L^2_\xi}\\
            &\lesssim \sum_{j + k = 2m} N^j \big|\big| (\nabla^{j + 1}_\xi \sfA)_{\frac1N} \nabla^k_\xi \psi \big|\big|_{L^2_\xi} \\
            &\lesssim N^\sigma,
    \end{align*} 
applying Cauchy-Schwartz in the first line, Plancharel's theorem in the second, and the derivative bounds on the symbol \eqref{groupvelocity} in the fourth line. This completes the proof of \eqref{kernel} and thereby the theorem. 
    
\begin{remark}
    In the proof above, we only needed the derivative bounds \eqref{groupvelocity} up to order $2m > \tfrac{d}{2}$. In fact, we can further relax the assumptions on the dispersion relation in Theorem \ref{thm:lwp} to \eqref{kernel}, though for applications to the standard examples \eqref{groupvelocity} is sufficient. 
\end{remark}

\section{Non-linear estimates}\label{sec:nonlinear}

As a primer for the local well-posedness argument, we find it instructive to prove \textit{a priori} estimates for the $\ell^\infty H^s$-norm of a solution and the $\ell^\infty H^0$-norm of the difference of two solutions. While the estimates themselves do not imply well-posedness, we will frequently refer back to their proofs over the course of the argument, which we defer to Sections \ref{sec:existence} and \ref{sec:cts}.

\subsection{Non-linear $\ell^\infty H^s$-bounds} \label{subsec:nonlin}

By the energy method, one can show the classical energy estimate for \eqref{dispersive} in the standard $H^s$-norms, 
	\begin{equation}\label{eq:classical}
		||u(t)||_{H^s}^2 \leq ||u_0||_{H^s}^2 \exp\left( \int_0^t C\Big( ||u (t') ||_{L^\infty_x}, ||\nabla u (t') ||_{L^\infty_x} \Big) \, dt'\right).
	\end{equation}
It follows from the Sobolev embedding inequality that we can close the energy estimate provided that $s > \tfrac{d}2 + 1$, which is the threshold for the classical well-posedness of derivative non-linear equations such as \eqref{dispersive}. For equations on the line, see for example \cite{BonaScott1976, TsutsumiFukuda1980,AbdelouhabEtAl1989}. We claim that the analogous result holds for \eqref{dispersive} on $\ell^\infty H^s$-spaces. 

\begin{theorem}[Energy estimate]\label{thm:apriori}
	Let $s > 0$, and suppose $u \in C^0_t (\ell^\infty H^s \cap \ell^\infty H^{\frac{d}{2}+ 1 +})_x ([0, T] \times \R^d)$ is a solution to \eqref{dispersive}, then 
		\begin{equation} \label{eq:apriori}
			||u (t)||_{\ell^\infty H^s}^2 \leq ||u_0||_{\ell^\infty H^s}^2 \exp \left( \int_0^t C\Big( ||u (t') ||_{L^\infty_x}, ||\nabla u (t') ||_{L^\infty_x} \Big) \, dt'\right),
		\end{equation}
	where $C( ||u||_{L^\infty_x}, ||\nabla u||_{L^\infty_x} )$ is some polynomial with non-negative coefficients. When $s = 0$, the analogous estimate holds after replacing the $L^\infty_x$-norms with $\ell^\infty H^{\frac{d}2+}$-norms. 
\end{theorem}

Continuing in the analogy with the classical energy estimate \eqref{classical}, we proceed by the energy method. We separate the non-perturbative terms from the perturbative ones, treating the former using integration-by-parts, and the latter by Cauchy-Schwartz. The estimate \eqref{apriori} follows then from Gronwall's inequality. We detail the case $s > 0$; the case $s = 0$ is similar. 

As a remark on notation, henceforth we shall suppress the dependence of the non-linearity on the complex conjugate of $u$, writing $\cQ(u) := \cQ(|u|^2)$ and $\cN(u) := \cN(u, \overline u)$. We will use $C$ to denote an implicit constant which grows polynomially in its arguments. 

\subsubsection*{\underline{Deriving the paradifferential equation}}

Applying the projection $P_N$ to the equation gives 
	\begin{equation}\label{eq:paradiff}
		\partial_t u_N + \sfA(\nabla) u_N + P_N \big( \cQ (u) \nabla u\big) = \cN(u)_N. 
	\end{equation}
The derivative non-linearity is the main obstruction to closing the estimate due to the possible derivative loss. To isolate the worst interactions, we perform a paraproduct decomposition, writing 
	\begin{align*}
		P_N \big( \cQ(u)\nabla u \big)
			&= P_N \left(  \cQ(u)_{\leq \frac{N}{128}}  \nabla u  \right) + P_N \sum_{A > \frac{N}{128}}  \cQ(u)_A \nabla u \\
			&= \cQ (u)_{\leq \frac{N}{128}}  \nabla u_N + [\cQ (u)_{\leq \frac{N}{128}}  , P_N] \nabla \widetilde{P_N} u + P_N \Big(\widetilde{P_N} \cQ(u) \, \nabla u_{\leq\frac{N}{128}}\Big) + \sum_{\substack{A \sim B \\ A, B > \frac{N}{128}}} P_N \Big( \cQ(u)_A \, \nabla u_B\Big).
	\end{align*}
We leave the first term on the left-hand side of \eqref{paradiff}, and regard the rest as perturbative by placing them on the right-hand side. This furnishes a system of non-linear equations for each frequency-localised piece $u_N$ evolving on a low frequency background. 

To obtain the equation for $\chi u_N$, we simply multiply \eqref{paradiff} by the cut-off $\chi$ and commute the cut-off into the non-perturbative terms, furnishing  
	\begin{equation}\label{eq:paradiff2}
		\mathsf L u_N + \mathsf T_{\cQ(u)} \nabla u_N =  \chi \cN (u)_N -  \cB_N (\cQ(u),u) , 
	\end{equation}
where the non-perturbative terms on the left consist of the linear terms $\mathsf L u_N$ and the low-high interactions $\mathsf T_{\cQ (u)} \nabla u_N$ in the case when the derivative falls on the high frequencies, 
	\begin{align}
		\mathsf L u_N
			&:= \partial_t \big( \chi u_N \big) + \sfA(\nabla) \big( \chi u_N \big) + [\chi, \sfA(\nabla)] u_N  \label{eq:linear},\\
		\mathsf T_{\cQ(u)} \nabla u_N
			&:=  \cQ(u)_{\leq \frac{N}{128}} \nabla \big( \chi u_N \big) \label{eq:lohiparadiff} ,
	\end{align}
and the perturbative terms on the right consist of the lower-order non-linearity $\cN(u)$ and a bilinear expression $\cB_N (\cQ(u), u)$ containing the low-high interactions in the case where the derivative does not fall on the high frequencies, the high-low interactions, and the high-high interactions, 
	\[
		\cB_N (\cQ(u), u)
			:= \text{low-high} + \text{high-low} + \text{high-high}
	\]
where
	\begin{align}		
		\text{low-high}
			&:= \cQ(u)_{\leq \frac{N}{128}} [\chi, \nabla] u_N +  \chi [\cQ (u)_{\leq \frac{N}{128}}  , P_N] \nabla \widetilde{P_N} u,\label{eq:lohiparadiff2} \\
		\text{high-low}
			&:= \chi P_N \Big(\widetilde{P_N} \cQ(u) \, \nabla u_{\leq\frac{N}{128}}\Big), \label{eq:hiloparadiff} \\
		\text{high-high}
			&:=  \chi \sum_{\substack{A \sim B \\ A, B > \frac{N}{128}}} P_N \Big( \cQ(u)_A \, \nabla u_B\Big). \label{eq:hihiparadiff}		
	\end{align}	

\subsubsection*{\underline{Integrating the non-perturbative terms by parts}}	

As in the proof of the linear propagator bound \eqref{linearenergy}, we multiply the paradifferential equation \eqref{paradiff2} by $\overline{\chi u_N}$, take the real part, and integrate on the space-time region $[0, t] \times \R^d$. The terms on the left-hand side are treated by integration-by-parts.

\subsubsection*{Linear flow \eqref{linear}}

We treat the terms arising from the linear flow \eqref{linear} exactly as in the proof of \eqref{linearenergy}.

\subsubsection*{Low-high interactions \eqref{lohiparadiff}}

The most problematic term arises when the derivative falls on high frequencies. In this case, one naively expects to see derivative-loss when trying to close the estimate. Nevertheless, in the spirit of the classical energy estimates, we can place the derivative on the low frequencies, which \textit{is} favourable for closing the \textit{a priori} estimate, via integration-by-parts,
	\begin{align*}
		\int_{\R^d} \Re \overline{\big( \chi u_N \big)} \left( \cQ(u)_{\leq \frac{N}{128}} \nabla \big( \chi u_N \big) \right) \, dx 
			&= \int_{\R^d}  \cQ(u)_{\leq \frac{N}{128}}\nabla \left( \frac12 \big| \chi u_N \big|^2 \right) \, dx \\
			&= -\frac12 \int_{\R^d} \nabla  \cQ (u)_{\leq \frac{N}{128}} \big| \chi u_N \big|^2 \, dx.
	\end{align*}
Placing the low frequency term in $L^\infty_x$, we conclude 
	\begin{align*}
		\left|\int_{\R^d} \Re \overline{\big( \chi u_N \big)} \left( \cQ(u)_{\leq \frac{N}{128}} \nabla \big( \chi u_N \big) \right) \, dx \right|
			&\lesssim ||\nabla \cQ(u)||_{L^\infty_x} ||\chi u_N||_{L^2_x}^2 \\
			&\lesssim C\Big(||u||_{L^\infty_x},  ||\nabla u ||_{L^\infty_x} \Big) ||u_N||_{(\ell^\infty L^2)_N}^2. 
	\end{align*}

\subsubsection*{\underline{Estimating the perturbative terms}}	

	Continuing from the previous section, it remains to treat the right-hand side of \eqref{paradiff2}. The next moves we take in the energy method are to take the supremum in $j \in \Z^d$, multiply by $N^{2s}$, and sum in $N \in 2^\N$. The right-hand side is estimated then by Cauchy-Schwartz,
	\[
		\sum_{N \in 2^\N} N^{2s} \sup_{j \in \Z^d} \Big| \int_{\R^d} \Re \overline{\big( \chi u_N \big)}  \big(\text{R.H.S. of \eqref{paradiff2}}\big)  dx \Big| \leq ||u||_{\ell^\infty H^s} \Big|\Big| N^s \sup_{j \in \Z^d} ||\text{R.H.S. of \eqref{paradiff2}}||_{L^2_x} \Big|\Big|_{\ell^2_N}. 
	\]
Thus, we want to estimate the perturbative terms in what is essentially the $\ell^\infty H^s$-norm. To apply Gronwall's inequality, we aim for a bound of the form 
	\[
		||\text{R.H.S. of \eqref{paradiff2}}||_{L^2_x}
			\lesssim  C\Big(||u||_{L^\infty_x},  ||\nabla u ||_{L^\infty_x} \Big) ||u_N||_{(\ell^\infty L^2)_N} 
	\]
or the weaker bound 
	\[
		\Big|\Big| N^s \sup_{j \in \Z^d} ||\text{R.H.S. of \eqref{paradiff2}}||_{L^2_x} \Big|\Big|_{\ell^2_N}
			\lesssim C\Big(||u||_{L^\infty_x},  ||\nabla u ||_{L^\infty_x} \Big)  ||u||_{\ell^\infty H^s}. 
	\]

	\subsubsection*{Low-high interactions \eqref{lohiparadiff2}}
	
	In the first term, the derivative falls on the cut-off $[\nabla, \chi] = \nabla \chi$, which again we recall contributes an amplitude $|\nabla \chi| \leq N^{-2} \leq 1$ and is localised in space to neighboring intervals of size $|I| \approx N^2$. Placing the low frequency term in $L^\infty_x$ and the high frequency terms in $L^2_x$, we obtain 
		\begin{align*}
			\Bigl|\Bigl|  \cQ(u)_{\leq \frac{N}{128}} [\chi, \nabla] u_N \Bigr|\Bigr|_{L^2_x}
				&\leq  || \cQ(u) ||_{L^\infty_x} \Big|\Big| N^{-2} \nabla \chi \big( \tfrac{x - j}{N^2}\big) u_N \Big|\Big|_{L^2_x} \\
				&\lesssim C\Big(||u||_{L^\infty_x}\Big) \Big(N^{-2} \sum_{k = j + O(1)} ||\chi_{k, N} u_N ||_{L^2_x} \Big) \\
				&\lesssim C\Big(||u||_{L^\infty_x}\Big)  ||u_N||_{(\ell^\infty L^2)_N} .
		\end{align*}
	For the second term, the Littlewood-Paley product rule \eqref{commutator} and the Sobolev-Bernstein inequality \eqref{sobolevbernstein} allow us to move the derivative from the high frequency factor to the low frequency factor, 
		\begin{align*}
			\Big|\Big| [\cQ (u)_{\leq \frac{N}{128}}  , P_N] \nabla \widetilde{P_N} u\Big|\Big|_{(\ell^\infty L^2)_N}
				&\lesssim  || \nabla \cQ(u)||_{L^\infty_x} || N^{-1} \nabla \widetilde{P_N} u||_{(\ell^\infty L^2)_N}\\
				&\lesssim C\Big(||u||_{L^\infty_x},  ||\nabla u||_{L^\infty_x}\Big)  ||\widetilde{P_N} u||_{(\ell^\infty L^2)_N}.
			\end{align*}

		\subsubsection*{High-low interactions \eqref{hiloparadiff}}

		Using boundedness of $P_N$ a l{\'a} \eqref{bounded} and placing the low frequency term in $L^\infty_x$, we obtain 
			\begin{align*}
				\Big|\Big|  P_N \Big(\widetilde{P_N} \cQ(u) \, \nabla u_{\leq\frac{N}{128}}\Big) \Big|\Big|_{(\ell^\infty L^2)_N}
					&\leq || \nabla u||_{L^\infty_x} || \widetilde{P_N} \cQ(u) ||_{(\ell^\infty L^2)_N}.
			\end{align*}
		Multiplying by $N^s$ and square summing in $N \in 2^\N$ furnishes the $\ell^\infty H^s$-norm of $\cQ(u)$ on the right hand side, which we can estimate via the chain rule \eqref{chain1},
			\begin{align*}
				|| \nabla u||_{L^\infty_x} \Big|\Big| N^s || \widetilde{P_N} \cQ(u) ||_{(\ell^\infty L^2)_N}\Big|\Big|_{\ell^2_N}
					&\sim  ||\nabla u||_{L^\infty_x}  ||\cQ(u)||_{\ell^\infty H^s} \\
					&\lesssim C\Big(||u||_{L^\infty_x},  ||\nabla u||_{L^\infty_x}\Big) ||u||_{L^\infty_x}^{p - 1} ||u||_{\ell^\infty H^s} . 
			\end{align*}

\subsubsection*{High-high interactions \eqref{hihiparadiff}}

	By Cauchy-Schwartz in $x \in \R^d$ and the triangle inequality, 
	\begin{align*}
		\bigg| \int_{\R^d} \big( \chi u_N \big) \Big( \chi P_N  \sum_{\substack{A \sim B \\ A, B > \frac{N}{128}}}  \cQ(u)_A \nabla u_B\Big)  \, dx \bigg| 
			&\leq ||u_N||_{(\ell^\infty L^2)_N} \Big|\Big| P_N  \sum_{\substack{A \sim B \\ A, B > \frac{N}{128}}}  \cQ(u)_A \nabla u_B  \Big|\Big|_{(\ell^\infty L^2)_N}\\
			&\lesssim ||u_N||_{(\ell^\infty L^2)_N} ||\nabla u||_{L^\infty_x} \sum_{A > \frac{N}{128}} ||\cQ(u)_A ||_{(\ell^\infty L^2)_A},
	\end{align*}
where in the second line we have placed $\nabla u$ in $L^\infty_x$ and, using equivalence of norms \eqref{monotone}, replaced the $(\ell^\infty_j L^2_x)_N$-norms with $(\ell^\infty_j L^2_x)_A$-norms. Multiplying by $N^{2s}$ and summing in $N \in 2^\N$, 
	\begin{align*}
		&\sum_{N \in 2^\N} \Big( N^s||u_N||_{(\ell^\infty L^2)_N}\Big) \Big( N^s \sum_{A > \frac{N}{128}} ||\cQ(u)_A ||_{(\ell^\infty L^2)_A}\Big)
			\\
			&\qquad = \sum_{C > \frac{1}{128}} C^{-s} \sum_{N \in 2^\N} \Big( N^s ||u_N||_{(\ell^\infty L^2)_N}\Big)  \Big( (CN)^s ||\cQ(u)_{CN} ||_{(\ell^\infty L^2)_{CN}}\Big)\\
			&\qquad \lesssim C\Big(||u||_{L^\infty_x}\Big) ||u||_{\ell^\infty H^s}^2 ,
	\end{align*}
where in the second line we make the change of variables $CN = A$ and interchange the sums, and in the third line we apply Cauchy-Schwartz and the chain rule \eqref{chain1}.

\subsubsection*{Lower-order non-linearity}

After all our moves, this term becomes precisely the $\ell^\infty H^s$-norm of $\cN(u)$. The chain rule \eqref{chain1} implies
	\[
		||\cN(u)||_{\ell^\infty H^s}
			\lesssim C\Big( ||u||_{L^\infty_x}\Big) ||u||_{\ell^\infty H^s}^2. 
	\]

\subsubsection*{\underline{Concluding the estimate}} 

Recall the sequence of moves we took consisted of multiplying the paradifferential equation \eqref{paradiff} by $\overline{\chi u_N}$, taking the real part, integrating on the space-time region $[0, t] \times \R^d$, taking the supremum in $j \in \Z^d$, multiplying by $N^{2s}$, and summing in $N \in 2^\N$. Collecting the resulting expressions for the non-perturbative terms on the left-hand side and the estimates for the perturbative terms on the right-hand side, we obtain 	
	\[
		 ||u (t) ||_{\ell^\infty H^s}^2 \leq  ||u_0||_{\ell^\infty H^s}^2 + \int_0^t  C\Big( ||u (t') ||_{L^\infty_x}, ||\nabla u(t')||_{L^\infty_x} \Big) ||u (t')||_{\ell^\infty H^s}^2   \, dt'. 
	\]
We conclude the energy estimate \eqref{apriori} by Gronwall's inequality. 

\subsection{Non-linear $\ell^\infty H^0$-difference bounds}\label{subsec:difference}

We would like to show continuous dependence of solutions to \eqref{dispersive} on the initial data with respect to the $\ell^\infty H^s$-topology for $s > \tfrac{d}2 + 1$. It will be convenient to start by showing a weaker statement, namely Lipschitz continuity with respect to the $\ell^\infty H^0$-topology. 

\begin{theorem}[$\ell^\infty H^0$-Lipschitz continuity]
	Let $u, v \in C^0_t (\ell^\infty H^{\frac{d }{2} + 1 +})_x ([0, T]\times \R^d)$ be solutions to \eqref{dispersive}, then
		\begin{equation}\label{eq:lipschitz}
			||u(t) - v(t)||_{\ell^\infty H^0}^2 \leq ||u_0 - v_0||_{\ell^\infty H^0}^2 \exp \Big( \int_0^t C \Big( ||u(t')||_{\ell^\infty H^{\frac{d}{2}+ 1 +}} , ||v(t')||_{\ell^\infty H^{\frac{d}{2}+ 1 +}} \Big)  \, dt' \Big),
		\end{equation}
	where $C ( ||u(t')||_{\ell^\infty H^{\frac{d}{2}+1 +}} , ||v(t')||_{\ell^\infty H^{\frac{d}2+ 1 +}})$ is some polynomial with non-negative coefficients. 
\end{theorem}

Subtracting the respective equations of two solutions $u, v \in L^\infty_t (\ell^\infty H^{\frac{d }{2} + 1 +})_x ([0, T]\times \R^d)$ to \eqref{dispersive}, we see that their difference satisfies 
\begin{equation}\label{eq:differenceeq}
	\begin{split}
	\partial_t (u - v) + \sfA(\nabla) (u - v) + \cQ(u) \nabla (u - v)  
		&=\big( \cN(u) - \cN(v) \big)+ \big(\cQ(v) -\cQ(u)\big)\nabla v,\\
	(u - v)_{|t = 0}
		&= u_0 - v_0. 	
	\end{split}
\end{equation}
Following the proof of the linear \eqref{linearenergy} and non-linear estimates \eqref{apriori}, we have good bounds arising from the terms on the left-hand side. To bound the terms on the right-hand side, we use the endpoint product estimate \eqref{product0} and the Lipschitz bounds \eqref{chain3},
\begin{align*}
	||(\cQ(u) - \cQ(v)) \nabla v||_{\ell^\infty H^0} 
		&\lesssim  ||\nabla v||_{\ell^\infty H^{\frac{d}2+}} ||\cQ(u) - \cQ(v)||_{\ell^\infty H^0} \\
		&\lesssim  C \Big( ||u||_{\ell^\infty H^{\frac{d}{2} +}} , ||v||_{\ell^\infty H^{\frac{d}{2}+ 1 +}} \Big)   ||u - v||_{\ell^\infty H^0}. 
\end{align*}	
and	
\begin{align*}
	||\cN(u) - \cN(v)||_{\ell^\infty H^0} 
		&\lesssim C \Big( ||u||_{\ell^\infty H^{\frac{d}{2} +}} , ||v||_{\ell^\infty H^{\frac{d}{2}+}} \Big) ||u - v||_{\ell^\infty H^0}.
\end{align*}
This completes the proof. 

\section{Existence of a solution}\label{sec:existence}

Since $\ell^\infty H^s (\R^d)$ for $s > \tfrac{d}{2}$ is a Banach algebra, \eqref{product}, and preserved by the linear flow, \eqref{linearenergy}, we can apply the semigroup method of Kato \cite{Kato1993} to prove local well-posedness for a large class of non-linear perturbations of the linear equation \eqref{lindispersive}. This will be our basic starting point for constructing solutions to the non-linear equation \eqref{dispersive}. 

\begin{lemma}[Abstract local well-posedness scheme]\label{lem:iteration}
	Consider the initial data problem for the non-linear dispersive equation 
	\begin{equation}\label{eq:general}
		\begin{split}
			\partial_t u + \sfA(\nabla) u 
				&= \cB(u),\\
			u_{|t = 0}
				&= u_0,
		\end{split}
	\end{equation}
	where, for $s \geq 0$, the non-linear operator $\cB: \ell^\infty H^s (\R^d) \to \ell^\infty H^s (\R^d)$ satisfies local Sobolev bounds 
	\begin{equation}\label{eq:local1}
		||\cB(u)||_{\ell^\infty H^s} 
			\lesssim_{R} ||u||_{\ell^\infty H^s}, 
	\end{equation}
	and local Lipschitz bounds, 
	\begin{equation}\label{eq:local2}
		||\cB(u) - \cB(v)||_{\ell^\infty H^s}
			\lesssim_{R} ||u - v||_{\ell^\infty H^s}
	\end{equation}
	for $||u||_{\ell^\infty H^s}, ||v||_{\ell^\infty H^s} \leq R$. Then \eqref{general} is locally well-posed in $\ell^\infty H^s (\R^d)$. In fact, the data to solution map $u_0 \mapsto u$ is locally Lipschitz continuous. 
\end{lemma}

\begin{proof}
	The proof is standard; we detail that of local existence, leaving uniqueness and Lipschitz continuous dependence on data as exercises. Consider the Duhamel integral formulation of the equation 
	\begin{equation}\label{eq:duhamel}
		u(t) 
			= e^{-t \sfA(\nabla)} u_0 + \int_0^t e^{-(t - t') \sfA(\nabla)} \cB(u(t')) \, ds. 
	\end{equation}
	From the perspective of the Duhamel formula \eqref{duhamel}, $u \in C^0_t (\ell^\infty H^s)_x([0, T] \times \R^d)$ is a solution to \eqref{general} if it is a fixed point of the map  
		\[
			\Phi(u)(t)
				:= e^{-t \sfA(\nabla)} u_0 + \int_0^t e^{-(t - t') \sfA(\nabla)} \cB(u(t')) \, dt'.
		\]	
	We argue by the contraction mapping principle, showing that $\Phi$ is a contraction mapping on the closed ball $B_{R} \subseteq C^0_t (\ell^\infty H^s)_x ([0, T] \times \R^d)$ with radius $R := 2 ||u_0||_{\ell^\infty H^s}$ and a small time $0 < T \ll 1$ to be chosen later. By Minkowski's inequality, the linear propagator bounds \eqref{linearenergy} and the local Sobolev bounds \eqref{local1}
		\begin{align*}
			||\Phi(u) (t)||_{\ell^\infty H^s}
				&\leq ||u_0||_{\ell^\infty H^s} \exp (O(t)) + ||\cB(u)||_{C^0_t (\ell^\infty H^s)_x} \int_0^t \exp(O(t - t')) \, dt'\\
				&\leq||u_0||_{\ell^\infty H^s} \exp (O(t))+ O_R \Big(T ||u||_{C^0_t (\ell^\infty H^s)_x}\Big) \leq R,
		\end{align*}
	taking $T \ll_R 1$ sufficiently small to defeat the implicit constants. This shows that $\Phi$ maps the ball to itself $\Phi : B_R \to B_R$. To show that $\Phi$ is a contraction, a similar argument using in addition the local Lipschitz bounds \eqref{local2} furnishes 
		\begin{align*}
			||\Phi(u)(t) - \Phi(v)(t)||_{\ell^\infty H^s}
				&\leq  ||\cB(u) - \cB(v)||_{C^0_t (\ell^\infty H^s)_x}  \int_0^t \exp(O(t - t')) \, dt' \\
				&\leq O_R \Big( T ||u - v||_{C^0_t (\ell^\infty H^s)_x}\Big) \leq \frac12 ||u - v||_{C^0_t (\ell^\infty H^s)_x}. 
		\end{align*}
	By the contraction mapping principle, $\Phi$ admits a fixed point, completing the proof.  
\end{proof}

It follows then from the chain rule \eqref{chain1}-\eqref{chain2} and Sobolev embedding \eqref{sobolev} that the equation \eqref{dispersive} without derivative non-linearity, i.e. $\cQ(u) \equiv 0$, is locally well-posed in $\ell^\infty H^s (\R^d)$ for $s > \tfrac{d}{2}$ with Lipschitz dependence on the initial data. Of course, when the derivative in the non-linearity is present, $\cQ(u) \neq 0$, one has to work a little harder. This is the subject of the remainder of Section \ref{sec:existence}. 

\subsection{Constructing approximate solutions}

If one naively attempts to implement the iteration scheme from Lemma \ref{lem:iteration}, one quickly sees that the derivative non-linearity $\cQ (u) \nabla u$ incurs a loss of regularity when trying to prove bounds of the form \eqref{local1}-\eqref{local2}. Instead, we can implement Lemma \ref{lem:iteration} to construct solutions $u^{(M)} \in C^0_t (\ell^\infty H^s)_x ([0, T^{(M)}] \times \R^d)$ for $s > \tfrac{d}{2} + 1$ to the regularised equation 
\begin{equation}\label{eq:galerkin}
	\begin{split}
	\partial_t u^{(M)} + \sfA(\nabla) u^{(M)} + \cQ(u^{(M)})  \nabla \Big( P_{\leq M} P_{\leq M} u^{(M)} \Big) 
		&= \cN(u^{(M)}),\\
	u^{(M)}_{|t = 0}
		&= u_0.
	\end{split}
\end{equation}
We recover the original equation \eqref{dispersive} formally taking $M \to \infty$ in the regularised equation \eqref{galerkin}. Thus $\{u^{(M)}\}_M \subseteq C^0_t (\ell^\infty H^s)_x$ forms a sequence of approximate solutions which we want to show converge (in a weaker topology, e.g. $C^0_t (\ell^\infty H^0)_x$) to a solution $u \in C^0_t (\ell^\infty H^s)_x$ to the original equation \eqref{dispersive}. 

To show \eqref{galerkin} is well-posed, observe that the non-linearity  
\[
	\cB(u^{(M)}) := \cN(u^{(M)}) - \cQ(u^{(M)})  \nabla \Big(P_{\leq M} P_{\leq M} u^{(M)} \Big)
\]
satisfies \eqref{local1}-\eqref{local2} and so we can directly apply the Picard iteration argument of Lemma \ref{lem:iteration}. Indeed, introducing the frequency cut-offs\footnote{The choice of regularisation $P_{\leq M} P_{\leq M}$ might seem excessive at first glance, though as we shall see in the sequel it will be convenient for proving energy estimates. } $P_{\leq M}$ allows us to avoid the derivative-loss in the estimates at the cost of the implicit constants growing linearly in $M$, 
and thus the maximal time of existence $T^{(M)}$ possibly shrinking in $M$. Applying the product bounds \eqref{product}, difference bounds \eqref{chain2}, Sobolev embedding \eqref{sobolev}, Sobolev-Bernstein \eqref{sobolevbernstein}, and boundedness of the projections \eqref{bounded},
	\begin{align*}
		||\cB(u) - \cB(v)||_{\ell^\infty H^s}
			&\leq ||\cN(u) - \cN(v)||_{\ell^\infty H^s} \\
			&\qquad + || \big(\cQ(u) - \cQ(v)\big)  \nabla\big( P_{\leq M} P_{\leq M} u\big) ||_{\ell^\infty H^s} + ||\cQ(u)\nabla P_{\leq M} P_{\leq M} \big(u - v \big)||_{\ell^\infty H^s} \\
			&\lesssim  M \cdot  C\big(||u||_{\ell^\infty H^s}, ||v||_{\ell^\infty H^s}\big)  ||u - v||_{\ell^\infty H^s},
	\end{align*}
and clearly $\cB(0) \equiv 0$. 

\subsection{Uniform energy bounds}

While the $\ell^\infty H^s$-estimates arising from Lemma \ref{lem:iteration} fail to furnish any uniform bounds on the sequence as $M \to \infty$, we can instead appeal to the proof of the \textit{a priori} estimate \eqref{apriori} for the original equation to show 
	\begin{equation}\label{eq:aprioriM}
		||u^{(M)} (t) ||_{\ell^\infty H^s}^2 
			\leq ||u_0||_{\ell^\infty H^s}^2 \exp \Big( \int_0^t C \Big( ||u^{(M)}(t')||_{L^\infty_x}, || \nabla u^{(M)} (t')||_{L^\infty_x} \Big) \, dt' \Big). 
	\end{equation}
It would then follow from Sobolev embedding \eqref{sobolev} and a bootstrap argument that the approximate solutions $\{u^{(M)}\}_M$ can be continued up to a common time $T \sim C(||u_0||_{\ell^\infty H^s})^{-1}$ and satisfy the uniform energy estimate 
	\begin{equation}\label{eq:uniformM}
		||u^{(M)}||_{C^0_t (\ell^\infty H^s)_x([0, T]\times \R^d)} 
			\leq 2 ||u_0||_{\ell^\infty H^s}.
	\end{equation}
Let us proceed then to the proof of the \textit{a priori} estimate \eqref{aprioriM}. Using the notation \eqref{linear}-\eqref{hihiparadiff}, the paradifferential equation for the regularised equation \eqref{galerkin} is
	\begin{equation} \label{eq:paradiffgalerkin}
		\mathsf L u_N^{(M)} + \mathsf T_{\cQ(u^{(M)})} \nabla P_{\leq M} P_{\leq M} u_N ^{(M)}
			= \chi \cN(u^{(M)})_N - \cB_N (\cQ(u^{(M)}), P_{\leq M} P_{\leq M} u^{(M)}). 
	\end{equation}
where the key term to remember the precise form of is the low-high term 
	\[
		\mathsf T_{\cQ(u^{(M)})} \nabla P_{\leq M} P_{\leq M} u_N ^{(M)} = \cQ(u^{(M)})_{\leq \frac{N}{128}}  \nabla \Big( \chi P_{\leq M} P_{\leq M} u_N^{(M)}\Big).
	\]
Multiplying by $\overline{\chi u^{(M)}_N}$, taking the real part, and integrating, the first, third, and fourth terms can be dealt with exactly as in the proof of the original energy estimate \eqref{apriori}; the projections $P_{\leq M}$ are harmless as they are bounded operators uniformly in $M$. For the low-high term, we have $P_{\leq M}P_{\leq M} u_N \equiv u_N$ when $N \leq \tfrac{M}{128}$, in which case we proceed via the integration-by-parts argument as in the proof of \eqref{apriori}, while $P_{\leq M}P_{\leq M} u_N \equiv 0$ when $N \geq 128 M$. 

It remains then to estimate the low-high term in \eqref{paradiffgalerkin} when $N \sim M$. While the Littlewood-Paley projections are not true projections $P_{\leq N} P_{\leq N} u_N \neq u_N$, thanks to our judicious choice of regularisation $P_{\leq M} P_{\leq M}$ we can nevertheless recover a total derivative using self-adjointness of $P_{\leq M}$ and commuting,
		\begin{align*}
			\int_{\R^d} \overline{\big(\chi u_N^{(M)}\big)}\cQ(u^{(M)})_{\leq \frac{N}{128}}  \nabla \Big( \chi P_{\leq M} P_{\leq M} u_N^{(M)}\Big) \, dx 				
				&=\int_{\R^d} \overline{P_{\leq M} \big(\chi u_N^{(M)}\big) } \cQ(u^{(M)})_{\leq \frac{N}{128}} \nabla \Big( P_{\leq M} \big(\chi u^{(M)}_N\big) \Big) \, dx \\
				&\qquad + \int_{\R^d} \overline{\big(\chi u_N^{(M)}\big)}\cQ(u^{(M)})_{\leq \frac{N}{128}}  \nabla \Big( \big[\chi, P_{\leq M} P_{\leq M}\big] u_N^{(M)}\Big) \, dx \\
				&\qquad + \int_{\R^d} \overline{\big(\chi u_N^{(M)} \big)\big[\cQ(u^{(M)})_{\leq \frac{N}{128}}, P_{\leq M}\big]} \nabla \Big( P_{\leq M} \big( \chi u_N^{(M)}\big) \Big) \, dx\\
				&=: \mathrm{I} + \mathrm{II} + \mathrm{III}. 
		\end{align*}
	Taking the real part, we can integrate-by-parts to move the derivative in $\mathrm{I}$ onto the low frequency term, 
		\begin{align*}
			\Re \mathrm{I}
				&= \int_{\R^d} \cQ(u^{(M)})_{\leq \frac{N}{128}} \nabla \Big( \frac12 \big|P_{\leq M} \big(\chi u^{(M)}_N\big) \big|^2\Big)\, dx \\
				&= - \frac12 \int_{\R^d} \nabla \cQ(u^{(M)})_{\leq \frac{N}{128}}  \big|P_{\leq M} \big(\chi u^{(M)}_N\big) \big|^2 \, dx. 
		\end{align*}	
	Placing the low frequency term in $L^\infty_x$, 
		\begin{align*}
			|\Re\mathrm{I}|
				&\lesssim || \nabla \cQ(u^{(M)})||_{L^\infty_x} ||\chi P_{\leq M} u_N||_{L^2_x}^2 \\
				&\lesssim C\Big( ||u^{(M)}||_{L^\infty_x}, ||\nabla u^{(M)}||_{L^\infty_x} \Big) ||u^{(M)}||_{(\ell^\infty L^2)_N}. 
		\end{align*}
	The commutator terms $\mathrm{II}$ and $\mathrm{III}$ are similarly favourable. Indeed, by Cauchy-Schwartz, Sobolev-Bernstein \eqref{sobolevbernstein} and a variant of the commutator bound \eqref{commutator}, we have 
		\begin{align*}
			|\mathrm{II}|
				&\leq ||u^{(M)}_N||_{(\ell^\infty L^2)_N} ||\cQ(u^{(M)})_{\leq \frac{N}{128}}||_{L^\infty_x} N \Big|\Big| \big[ \chi, P_{\leq M} P_{\leq M} \big] u_N^{(M)} \Big|\Big|_{(\ell^\infty L^2)_N} \\
				&\lesssim C\Big( ||u^{(M)}||_{L^\infty_x} \Big) ||\nabla \chi||_{L^\infty_x}  ||u^{(M)}_N||_{(\ell^\infty L^2)_N}^2 
		\end{align*}
	and 
		\begin{align*}
			|\mathrm{III}|
				&\leq  ||u^{(M)}_N||_{(\ell^\infty L^2)_N} \Big|\Big| \big[\cQ(u^{(M)})_{\leq \frac{N}{128}}, P_{\leq M}\big] \nabla \Big( P_{\leq M} \big( \chi u_N^{(M)}\big)  \Big) \Big|\Big|_{(\ell^\infty L^2)_N} \\
				&\lesssim \big|\big| \nabla \cQ(u^{(M)})_{\leq \frac{N}{128}}\big|\big|_{L^\infty_x} ||u^{(M)}||_{(\ell^\infty L^2)_N}^2\\
				&\lesssim C\Big( ||u^{(M)}||_{L^\infty_x}, ||\nabla u^{(M)} ||_{L^\infty_x}\Big) ||u^{(M)}_N||_{(\ell^\infty L^2)_N}^2. 
		\end{align*}
	This completes the proof of \eqref{aprioriM}. 

	\subsection{Convergence to a solution}
	
	When working in $H^s$-spaces, the textbook approach, e.g. Taylor's treatment \cite[Chapter 15.7]{Taylor2011b}, is to extract a convergent subsequence via a compactness argument. Unfortunately, as remarked at the end of Section \ref{subsec:properties}, the analogue of the Rellich-Kondrachov compact embedding theorem fails in the setting of $\ell^\infty H^s (\R^d)$, so to extract a limit from the approximate solutions $\{ u^{(M)}\}_M$ we need to work directly with the equations for the differences $u^{(A)} - u^{(B)}$ and argue by energy estimates. We claim that $\{ u^{(N)}\}_N$ forms a Cauchy sequence in the $C^0_t (\ell^\infty H^0)_x$-topology,
		\begin{align}\label{eq:cauchy}
			\lim_{N, M \to \infty} ||u^{(N)} - u^{(M)}||_{C^0_t (\ell^\infty H^0)_x} = 0. 
		\end{align}
	Thus by completness there exists $u \in C^0_t (\ell^\infty H^0)_x ([0, T] \times \R^d)$ for which the approximate solutions converge to in the $C^0_t (\ell^\infty H^0)_x$-topology. In fact, interpolating with the uniform $C^0_t (\ell^\infty H^s)_x$-bounds, we see that $u \in C^0_t (\ell^\infty H^{s-})_x ([0, T] \times \R^d)$. Using the difference bounds one can check that this is indeed a distributional solution to the equation \eqref{dispersive}. 

	To prove \eqref{cauchy}, subtract the equation \eqref{galerkin} for $u^{(B)}$ from that of $u^{(A)}$, 
		\begin{equation}\label{eq:differenceM}
			\begin{split}
				\partial_t \big(u^{(A)} - u^{(B)}\big) + \sfA(\nabla) \big(u^{(A)} - u^{(B)} \big)  + \cQ(u^{(A)}) \nabla \Big( P_{\leq A} P_{\leq A} \big(u^{(A)} - u^{(B)}\big) \Big) 
					&= \text{R.H.S.}\\
				\big(u^{(A)} - u^{(B)}\big)_{|t = 0}
					&= 0. 	
			\end{split}
		\end{equation}
	where the right-hand side is 
		\begin{align*}
			\text{R.H.S.}
				&:= \cN(u^{(A)}) - \cN(u^{(B)}) \\
				&\qquad+\Big( \cQ(u^{(B)}) - \cQ(u^{(A)}) \Big) \nabla \Big( P_{\leq A} P_{\leq A} u^{(B)}\Big)\\
				&\qquad+ \cQ(u^{(B)}) \nabla \Big( P_{\leq B} P_{\leq B} u^{(B)} - P_{\leq A} P_{\leq A} u^{(B)} \Big) \\
				&=: \mathrm{I} + \mathrm{II} + \mathrm{III}.
		\end{align*}
	The left-hand side of the difference equation \eqref{differenceM} can be treated exactly as in the proof of the uniform energy bounds  \eqref{aprioriM} for the approximate solutions, modulo using the endpoint Coifman-Meyer product estimates \eqref{lowhigh1}-\eqref{product0}. On the right-hand side, we estimate the first two terms using the endpoint product estimate \eqref{product0} and $\ell^\infty H^0$-Lipschitz bounds \eqref{chain3}, 
		\begin{align*}
			||\mathrm{I} ||_{\ell^\infty H^0} 
				&\lesssim C \Big( ||u^{(A)}||_{\ell^\infty H^{\frac{d}{2} +}} , ||u^{(B)}||_{\ell^\infty H^{\frac{d}{2} +}}\Big) ||u^{(A)} - u^{(B)}||_{\ell^\infty H^0}\\
				&\lesssim C\Big(||u_0||_{\ell^\infty H^s}\Big) ||u^{(A)} - u^{(B)}||_{\ell^\infty H^0}
		\end{align*}
	and 
		\begin{align*}
			|| \mathrm{II} ||_{\ell^\infty H^0} 	
				&\lesssim ||\cQ(u^{(B)}) - \cQ(u^{(A)})||_{\ell^\infty H^0} ||u^{(B)}||_{\ell^\infty H^{\frac{d}{2} + 1 +}}
				\\
				&\lesssim C \Big( ||u^{(A)}||_{\ell^\infty H^{\frac{d}{2}+}} , ||u^{(B)}||_{\ell^\infty H^{\frac{d}{2} + 1 +}}\Big) ||u^{(A)} - u^{(B)}||_{\ell^\infty H^0}\\
				&\lesssim C\Big(||u_0||_{\ell^\infty H^s}\Big)  ||u^{(A)} - u^{(B)}||_{\ell^\infty H^0},
		\end{align*}
	where in the last lines we have used the uniform energy bound \eqref{uniformM}. For the third term on the right-hand side, observe that one of the terms is supported at high frequencies, so we can gain some smallness from the decay of high frequencies. More precisely, suppose without loss of generality $A \leq B$, then, writing $P_{\leq B} P_{\leq B} - P_{\leq A} P_{\leq A} = P_{> \frac{A}{128}}(P_{\leq B} P_{\leq B} - P_{\leq A} P_{\leq A})$, we can estimate 
		\begin{align*}
			||\mathrm{III}||_{\ell^\infty H^0} 
				&\lesssim ||\cQ(u^{(B)})||_{\ell^\infty H^0} || P_{> \frac{A}{128}} u^{(B)}||_{\ell^\infty H^{\frac{d}{2} + 1 + }}
				\\
				&\lesssim A^{0-} C \Big( ||u^{(B)}||_{\ell^\infty H^{\frac{d}{2} +}}\Big) ||u^{(B)}||_{\ell^\infty H^s}\\
				&\lesssim C\Big(||u_0||_{\ell^\infty H^s}\Big)  A^{0-} .
		\end{align*}
	Collecting the estimates above and applying Gronwall's inquality, we conclude 
		\[
			||u^{(A)} - u^{(B)}||_{C^0_t(\ell^\infty H^0)_x}^2 
				\lesssim C\Big(||u_0||_{\ell^\infty H^s}\Big)  A^{0-} \overset{A, B \to \infty}{\longrightarrow} 0 
		\]
	as desired. 

\section{Continuous dependence on data}\label{sec:cts}

To conclude local well-posedness for \eqref{dispersive}, it remains to show that the flow $t \mapsto u(t)$ in the phase space $\ell^\infty H^s (\R^d)$ is not only bounded but also continuous in time, and the data-to-solution map $u_0 \mapsto u$ is continuous with respect to the $\ell^\infty H^s$-topology. Our presentation closely follows the strategy of Ifrim-Tataru \cite[Section 5, 6]{IfrimTataru2022b}, and so we will keep the discussion brief. 

Before turning to the proofs, we need to introduce our main tool, the \textit{frequency envelope}, a bookkeeping device originally introduced by Tao in the context of wave maps \cite{Tao2001}. While the linear flow preserves frequency-localisation, different frequencies interact with one another under the non-linear flow, causing leakage of energy between nearby dyadic frequencies. To control this leakage, we define a $\delta$-frequency envelope $\{c_N\}_N \in \ell^2 (2^\N)$ for the initial data $u_0 \in \ell^\infty H^s (\R^d)$ by the following three properties:
	\begin{enumerate}
		\item energy bounds, 
			\begin{equation}
				N^s ||P_N u_0||_{(\ell^\infty L^2)_N} \leq c_N, 
            \end{equation}
		\item slowly varying,
			\begin{equation}\label{eq:slow}
				\frac{c_N}{c_M} \lesssim 2^{\delta|n - m|},
            \end{equation}
		\item sharpness,
			\begin{equation}
				||c_N||_{\ell^2_N} \sim ||u_0||_{\ell^\infty H^s}.
            \end{equation}
	\end{enumerate}	
To construct such an envelope, we take
	\[
		c_N := \max_{M \in 2^{\N}} 2^{- \delta |n - m|} M^s ||P_M u_0||_{(\ell^\infty L^2)_M}.
	\]	    
While it may be difficult to control the Littlewood-Paley piece of the solution $P_N u(t)$ by that of the initial data $P_N u_0$ due to the leakage between frequencies, we can hope to propagate the frequency envelope $\{c_N\}_N$ for the initial data $u_0$ to one for the solution $u(t)$ thanks to the slowly varying property \eqref{slow}. 

We will also need the following observations for the regularised data $\{P_{\leq N} u_0\}_N \subseteq C^\infty (\R^d)$,
	\begin{enumerate}
		\item uniform energy bounds,
			\begin{equation}\label{eq:envelopeuniform}
				||P_M P_{\leq N} u_0 ||_{\ell^\infty H^s} 
					\lesssim c_M, 	
			\end{equation}
		\item high frequency bounds, 
			\begin{equation}\label{eq:envelopehigh}
				||P_{\leq N} u_0 ||_{\ell^\infty H^{s + k}} 
					\lesssim N^k c_N,
			\end{equation}

		\item difference bounds, 
			\begin{equation}\label{eq:envelopediff}
				||P_{\leq 2N} u_0 - P_{\leq N} u_0||_{\ell^\infty H^0}
					\lesssim N^{-s} c_N.
			\end{equation}
		\item convergence 
			\begin{equation}\label{eq:convergence}
				\lim_{N \to \infty} ||P_{\leq N} u_0 - u_0||_{\ell^\infty H^s} = 0.
			\end{equation}
	\end{enumerate}

\subsection{Continuity in time}

Let $\{u^{(N)}\}_N$ denote the sequence of solutions to \eqref{dispersive} with regularised initial data $\{ P_{\leq N} u_0 \}_N$. Observe that the energy estimate \eqref{apriori} furnishes persistence of $\ell^\infty H^\infty$-regularity for solutions on a time interval $[0, T]$ depending only on the $\ell^\infty H^{s}$-norm of the initial data. In particular, $u^{(N)}$ are smooth functions of space-time. To establish continuity in time then for our original rough solution $u \in L^\infty_t (\ell^\infty H^s)_x ([0, T] \times \R^d)$, it suffices to show that 
	\[
		\lim_{N \to \infty}||u - u^{(N)}||_{L^\infty_t (\ell^\infty H^s)_x} = 0
	\]
since the smooth solutions are continuous in time $\{u^{(N)}\}_N \subseteq C^0_t (\ell^\infty H^s)_x ([0, T] \times \R^d)$. By the energy estimate \eqref{apriori} and the high frequency bounds \eqref{envelopehigh}, the regularised solutions satisfy the high frequency bounds 
	\begin{equation}\label{eq:regularhigh}
		||u^{(N)}||_{C^0_t (\ell^\infty H^{s + k})_x} 
			\lesssim N^k c_N,
	\end{equation}
and, using the Lipschitz bounds \eqref{lipschitz} and difference bounds \eqref{envelopediff}, 
	\begin{equation}\label{eq:regulardiff}
		||u^{(2N)} - u^{(N)}||_{C^0_t (\ell^\infty H^0)_x} 
			\lesssim N^{-s} c_N. 
	\end{equation}
Interpolating between the high frequency bounds \eqref{regularhigh} and difference bounds \eqref{regulardiff}, we obtain difference bounds at higher regularities, 
	\begin{equation}
		||u^{(2N)} - u^{(N)}||_{C^0_t (\ell^\infty H^s)_x}
			\lesssim  c_N. 
	\end{equation}
It follows as in \cite{IfrimTataru2022b} that 
	\[
		||u - u^{(N)}||_{C^0_t (\ell^\infty H^s)_x} 
			\lesssim ||c_M||_{\ell^2 (M \geq N)} \overset{N \to \infty}{\longrightarrow} 0. 
	\]

\subsection{Continuous dependence on data}

Interpolating between the energy estimates \eqref{apriori} and the difference estimates \eqref{lipschitz}, we can show that the flow is Lipschitz continuous with respect to the $\ell^\infty H^\sigma$-topologies for $0 \leq \sigma \leq s-$. However, we want to show the stronger statement that the data to solution map is continuous with respect to the $\ell^\infty H^s$-topology, that is, 
    \[
        \lim_{j \to \infty}||u - u_j||_{C^0_t (\ell^\infty H^s)_x} = 0 \qquad \text{whenever } \lim_{j \to \infty} ||u_0 - u_{0j}||_{\ell^\infty H^s} = 0
    \]
for solutions $u$ and $\{u_j\}_j$ corresponding to data $u_0$ and $\{u_{0j}\}_j$. Rather than directly comparing the solutions, it is convenient to use the regularised solutions $u_j^{(N)}$ and $u^{(N)}$ as proxies. 
	\begin{align*}
		||u - u_j ||_{C^0_t (\ell^\infty H^s)_x}
			&\leq ||u^{(N)} - u^{(N)}_j||_{C^0_t (\ell^\infty H^s)_x} + || u^{(N)}_j - u_j||_{C^0_t \ell^\infty H^s} + ||u - u^{(N)}||_{C^0_t (\ell^\infty H^s)_x}\\
			&\leq ||u^{(N)} - u^{(N)}_j||_{C^0_t (\ell^\infty H^s)_x}  + ||c_M^j||_{\ell^2 (M \geq N)} + ||c_M||_{\ell^2 (M \geq N)}.
	\end{align*}
We can take the frequency envelopes such that $\{c_M^j\}_M \to \{c_M\}_M$ in $\ell^2 (2^\N)$. Thus, taking $j \to \infty$ and then $N \to \infty$, we conclude the result. 

\section{Propagation of almost periodicity}\label{sec:periodic}

As an application of our proof of local well-posedness, we show Theorem \ref{thm:AP}, that almost periodicity persists under the non-linear evolution, and that the spectrum remains within the integer span of the initial spectrum. 

Let $\Lambda \subseteq \R$ be a countable set of frequencies forming a $\Z$-module, i.e. closed under addition, subtraction, and contains the zero frequency. For example, the integer span of a spectrum forms a $\Z$-module. Define the set of trigonometric polynomials with frequencies in $\Lambda$ by 
	\[
		\cT_\Lambda (\R)
			:= \Big\{ \sum_{\lambda \in \Lambda} a_\lambda e^{i \lambda x} : \text{finitely-many non-zero $a_\lambda \in \C$}  \Big\}.
	\]
We define the space of \emph{(Bohr) almost periodic functions} with spectrum in $\Lambda$, denoted by $\mathtt{AP}_\Lambda (\R)$, as the closure of $\cT_\Lambda (\R)$ with respect to the uniform topology,  
\[
	\mathtt{AP}_\Lambda (\R) 
		:= \overline{\cT_\Lambda (\R)}^{L^\infty_x}.
\]
For $u \in \mathtt{AP}_\Lambda (\R)$, the \textit{spectrum},  
\[
	\sigma (u) 
		:= \Big\{ \xi \in \R : \lim_{L \to \infty} \frac{1}{2L} \int_{-L}^L u(x) e^{i x \xi} \, dx\neq 0 \Big\},
\] 
is well-defined, continuous with respect to the uniform topology, and satisfies $\sigma(u) \subseteq \Lambda$; see \cite[Chapter VI.5]{Katznelson2004}. 

When the frequency set is finitely-generated over the integers, i.e. $\Lambda = \omega \cdot \Z^n$ for some vector $\omega \in \R^n$, we say that the functions in $\mathtt{AP}_\Lambda (\R)$ are \textit{quasi-periodic}. The case where the frequency set is generated over the integers by a single real $\Lambda = \omega \Z$ for some $\omega \in \R$ concides with the usual notion of $\tfrac{2\pi}{\omega}$-\textit{periodicity}. If instead the frequency set is generated by a single real over the rationals, i.e. $\Lambda = \omega \Q$ for some $\omega \in \R$, we say that the functions in $\mathtt{AP}_\Lambda (\R)$ are \textit{limit periodic}. For further reading on these classes of almost periodic functions, we point the interested reader to the classical work of Besocovitch \cite{MR68029}. 

We will consider the class of almost periodic functions with $\ell^\infty H^s$-regularity, i.e. $(\mathtt{AP}_\Lambda \cap \ell^\infty H^s) (\R)$. For $s > \tfrac12$, it follows from Sobolev embedding \eqref{sobolev} that the $\ell^\infty H^s$-topology is stronger than the uniform topology, so it is in fact complete with respect to the $\ell^\infty H^s$-topology, 

\begin{lemma}[Completeness of almost periodic $\ell^\infty H^s$-functions]\label{lem:complete}
	For $s > \tfrac12$, the space of almost periodic functions with spectrum contained in $\Lambda$ and $\ell^\infty H^s$-regularity $(\mathtt{AP}_\Lambda \cap \ell^\infty H^s) (\R)$ forms a closed subspace of $\ell^\infty H^s (\R)$. 
\end{lemma}

Since every finite trigonometric polynomial is a smooth bounded function, $\cT_\Lambda (\R) \subseteq C^\infty (\R)$, it follows from construction of $\mathtt{AP}_\Lambda (\R)$ and Lemma \ref{lem:complete} that every almost periodic initial data $u_0 \in (\mathtt{AP}_\Lambda \cap \ell^\infty H^s) (\R)$ can be approximated in the $\ell^\infty H^s$-topology by finite trigonometric polynomials $\{u_{0j}\}_j \subseteq \cT_\Lambda (\R)$. By \eqref{ctsdependence} the continuous dependence on data from Theorem \ref{thm:lwp}, i.e. 
	\[
		\lim_{j \to \infty}||u - u_j||_{C^0_t (\ell^\infty H^s)_x} = 0 \qquad \text{as } \lim_{j \to \infty} ||u_0 - u_{0j}|| = 0,
	\]
and Lemma \ref{lem:complete}, it suffices to prove Theorem \ref{thm:AP} for solutions arising from finite trigonometric initial data. Our key observation is that the Picard iteration scheme, initialised by  data $u_0 \in \cT_\Lambda (\R)$, reduces to a solving finite systems of ODEs in the Fourier coefficients, since both the linear flow \eqref{lindispersive} and the non-linearity in \eqref{dispersive} preserve $\cT_\Lambda (\R)$, 

\begin{lemma}[Operations preserving $\cT_\Lambda (\R)$]
	Let $\Lambda \subseteq \R$ be a set of frequencies forming a $\Z$-module, and fix trigonometric polynomials $u, v \in \cT_\Lambda (\R)$ given by 
		\[
			u(x) = \sum_{\lambda \in \Lambda} a_\lambda e^{i \lambda x}, \qquad v(x) = \sum_{\mu \in \Lambda} b_\mu e^{i \mu x}. 
		\]
	The class of finite trigonometric polynomials $\cT_\Lambda (\R)$ is closed under the following operators:
	\begin{enumerate}
		\item complex conjugation, 
			\[
			\overline u(x) = \sum_{\lambda \in \Lambda} \overline{a_\lambda} e^{- i \lambda x},
			\]
		
		\item differentiation, 
			\[
				(\partial_x u)(x) 
					= \sum_{\lambda \in \Lambda} \lambda a_\lambda  e^{i \lambda x}, 
			\]
		
		\item \label{eq:convolve}convolution, 
			\[
				(\phi * u)(x) 
					= \sum_{\lambda \in \Lambda} \widehat{\phi} (\lambda) a_\lambda e^{i \lambda x},
			\]
		
		\item multiplication, 
			\[
				(uv)(x)
					= \sum_{\lambda,\mu \in \Lambda} a_\lambda b_\mu e^{i (\lambda + \mu) x}. 
			\]
	\end{enumerate}
	In view of \eqref{convolve}, $\cT_\Lambda (\R)$ is also closed under the action of Fourier multipliers, e.g. the linear propagator $e^{-t \sfA(\partial_x)}$ and the Littlewood-Paley projection $P_N$. 
\end{lemma}

\begin{proposition}[Abstract propagation of almost periodicity scheme] \label{prop:AP}
	Let $\Lambda \subseteq \R$ be a set of frequencies forming a $\Z$-module, and suppose $\cB: \ell^\infty H^s(\R^d) \to \ell^\infty H^s(\R^d)$ is a non-linear operator satisfying the hypotheses \eqref{local1}-\eqref{local2} of Lemma \ref{lem:iteration} and furthermore preserves $\cT_\Lambda (\R)$, i.e. $\cB(\cT_\Lambda (\R)) \subseteq \cT_\Lambda (\R)$. Then the solution to the initial data problem 
	\begin{equation}\label{eq:general2}
		\begin{split}
			\partial_t u + \sfA(\nabla) u 
				&= \cB(u),\\
			u_{|t = 0}
				&= u_0,
		\end{split}
	\end{equation}
	for almost periodic data $u_0 \in (\mathtt{AP}_\Lambda \cap \ell^\infty H^s) (\R)$ is also almost periodic $u \in C^0_t (\mathtt{AP}_\Lambda \cap \ell^\infty H^s)_x ([0, T] \times \R)$.
\end{proposition}

\begin{proof}
	Since the initial data problem for \eqref{general2} is continuous with respect to the $\ell^\infty H^s$-topology, it suffices to prove the result for initial data which is a finite trigonometric polynomial $u_0 \in \cT_\Lambda (\R)$. Implicit in the proof of local well-posedness for \eqref{general2} in Lemma \ref{lem:iteration}, the solutions are constructed via Picard iteration. That is, initialising by $u^{(0)} \equiv u_0$, we inductively construct solutions $\{ u^{(n)}\}_n \subseteq C^0_t (\ell^\infty H^s)_x ([0, T] \times \R)$ to the linear equation
		\begin{equation}\label{eq:iterate}
			\begin{split}
			\partial_t u^{(n + 1)} + \sfA(\nabla) u^{(n + 1)} 
				&= \cB(u^{(n)}),\\
			u^{(n + 1)}_{|t = 0}
				&= u_0.
			\end{split}	
		\end{equation}
	By the contraction mapping principle these iterates converge to the solution $u$ of \eqref{general2} in the $C^0_t (\ell^\infty H^s)_x$-topology. To conclude the result, it suffices to show that the iterates consist of finite trigonometric polynomials $\{u^{(n)}\}_n \subseteq C^0_t (\cT_\Lambda)_x ([0, T] \times \R)$ since $(\mathtt{AP}_\Lambda \cap \ell^\infty H^s) (\R)$ forms a closed subspace of $\ell^\infty H^s (\R)$.
	
	Suppose for induction that the $n$-th Picard iterate $u^{(n)} \in C^0_t (\cT_\Lambda)_x ([0, T] \times \R)$. Since the $(n + 1)$-th Picard iterate $u^{(n + 1)}$ satisfies the linear equation \eqref{iterate}, we can compute it explicitly via Duhamel's formula
		\begin{equation}\label{eq:duhamel2}
			u^{(n + 1)} (t, x)
				= e^{- t \sfA(\partial_x)} u_0 + \int_0^t e^{- (t - t') \sfA(\partial_x)} \cB(u^{(n)}(t')) \, dt'.
		\end{equation}
	Since both the linear propagator and the non-linearity preserve the class of finite trigonometric polynomials, $e^{-t \sfA(\partial_x)} (\cT_\Lambda (\R)) \subseteq \cT_\Lambda (\R)$ and $\cB(\cT_\Lambda (\R)) \subseteq \cT_\Lambda (\R)$, we can compute from Duhamel's formula \eqref{duhamel2} that the $(n + 1)$-the Picard iterate is a finite trigonometric polynomial $u^{(n + 1)} \in C^0_t (\cT_\Lambda)_x ([0, T] \times \R)$. 
\end{proof}

\begin{remark}
	If one keeps careful track of the Fourier coefficients for the Picard iterates, we can in fact show that they are finite trigonometric polynomials in space-time. 
\end{remark}

We are now ready to conclude Theorem \ref{thm:AP}. When $\cQ(u) \equiv 0$ in the equation \eqref{dispersive}, the result follows immediately from Proposition \ref{prop:AP}. In the derivative non-linearity case $\cQ(u) \not\equiv 0$, recall from Section \ref{sec:existence} that the solution $u$ to the equation \eqref{dispersive} was constructed as the limit of approximate solutions $u^{(M)}$ to the regularised equation \eqref{galerkin} which we recollect below, 
	\begin{align*}
		\partial_t u^{(M)} + \sfA(\nabla) u^{(M)} + \cQ(u^{(M)})  \nabla \Big( P_{\leq M} P_{\leq M} u^{(M)} \Big) 
			&= \cN(u^{(M)}),\\
		u^{(M)}_{|t = 0}
			&= u_0. 	
	\end{align*}
The non-linearity preserves $\cT_\Lambda (\R)$, so the solutions $u^{(M)}$ are almost periodic by Proposition \ref{prop:AP}. Interpolating between the uniform bounds \eqref{uniformM} and the $\ell^\infty H^0$-difference bounds \eqref{cauchy}, the approximate solutions converge in the $C^0_t (\ell^\infty H^{s-})_x$-topology to the solution $u$, so it is almost periodic by Lemma \ref{lem:complete}. 

\bibliographystyle{alpha}
\bibliography{external/biblio}

\end{document}